\documentclass[12pt]{article} %
\usepackage{amsmath,amssymb,amsthm}
\numberwithin{equation}{section} %
\theoremstyle{plain}
   \newtheorem{thm}{\hspace{\parindent}{\sc Theorem}}[section] %
   \newtheorem{pro}[thm]{\hspace{\parindent}Proposition}
   
   \newtheorem{lem}[thm]{\hspace{\parindent}Lemma}
\theoremstyle{remark} %
   \newtheorem{rem}{\hspace{\parindent}Remark}[section] %
\newcommand{\adots}{a = 0,\pm1,\pm2,\dots}
\newcommand{\Atxrho}{A(t,x;\rho)}
\newcommand{\Aktxrho}{A^{(k)}(t,x;\rho)}
\newcommand{\bR}{\mathbb{R}}
\newcommand{\Cinfty}{C^{\infty}}
\newcommand{\domain}{[0,T]\times \mathbb{R}^d}
\newcommand{\domains}{[0,T]\times \mathbb{R}^{2d}}
\newcommand{\Eta}{\mathcal{E}_t^0([0,T];B^a)}
\newcommand{\Etaprime}{\mathcal{E}_t^0([0,T];B'^a)}
\newcommand{\hdbar}{\hspace{0.08cm}\dbar}
\newcommand{\Hepsilon}{H_{\epsilon}}
\newcommand{\jdots}{j = 1,2,\dots,d}

\newcommand{\qepsilon}{q_{\epsilon}}
\newcommand{\Qepsilon}{Q_{\epsilon}}
\newcommand{\rittaire}{\text{\rm{Re\hspace{0.05cm}}}}
\newcommand{\Sspace}{\mathcal{S}(\bR^d)}
\newcommand{\sumjd}{\sum_{j=1}^d}
\newcommand{\txdx}{(t,X,D_x)}
\newcommand{\txxi}{(t,x,\xi)}
\newcommand{\tzdz}{(t,Z,D_z)}
\newcommand{\uepsilon}{u_{\epsilon}}
\newcommand{\utrho}{u(t;\rho)}
\newcommand{\Vtxrho}{V(t,x;\rho)}
\newcommand{\Vktxrho}{V_k(t,x;\rho)}
\newcommand{\Wijtxrho}{W_{ij}(t,x;\rho)}
\newcommand{\weight}{\left(<\xi>^2 + <x>^{2(M+1)}\right)}
\newcommand{\wtrho}{w(t;\rho)}
\newcommand{\zeo}{0 < \epsilon \leq 1}
\def\dbar{{\mathchar'26\mkern-12mud}}
\begin{document}

\title{On the Schr\"odinger equations with time-dependent potentials  growing  polynomially in the spatial direction}
\author{Wataru Ichinose\thanks{Corresponding author. This research is  partially supported by JSPS KAKENHI grant No.26400161 and Shinshu Univerisity. E-mail: ichinose@math.shinshu-u.ac.jp
}\ , \and Takayoshi Aoki\thanks{This research is  partially supported by Shinshu University RA grant No.25204513. E-mail: 13sm101h@shinshu-u.ac.jp}
}
\date{}
\maketitle
\begin{quote}
{\small Department of Mathematical Sciences, Shinshu University,
Matsumoto 390-8621, Japan}%
\end{quote}\par
\begin{abstract}
\noindent
 The  Cauchy problem for the Schr\"odinger equations is studied with time-dependent potentials  growing  polynomially in the spatial direction.  First the existence and the uniqueness of solutions are shown in the weighted Sobolev spaces.  In addition, we suppose that our potentials are depending on a parameter.  
Secondly  it is shown  that if potentials depend continuously and differentiably on the parameter,  the solutions to the Schr\"odinger equations respectively become  continuous and differentiable with respect to its parameter. 
 \end{abstract}
 {\bf Keywords} Schr\"odinger equations; time-dependent potentials; polynomially growing  potentials in the spatial direction; continuity and differentiability with respect to a parameter.
 \\
{\bf AMS Subject Classification (2010)}  35Q41; 35Q40.
\section{Introduction}
	Let $T > 0$ be an arbitrary constant.
	We will  study the Schr\"odinger equations
\begin{align} \label{1.1}
& i\hbar \frac{\partial u}{\partial t}(t)  = H(t)u(t)\notag\\ 
& := \left[ \frac{1}{2m}\sum_{j=1}^d
      \left(\frac{\hbar}{i}\frac{\partial}{\partial x_j} - qA_j(t,x)\right)^2 + qV(t,x)\right]u(t),
\end{align}
where $t \in [0,T], x = (x_1,\dotsc,x_d) \in \mathbb{R}^d$, $ \bigl(V(t,x),A(t,x)\bigr) = (V,A_1,A_2,\dotsc,A_d) \in \bR^{d+1}$ are  electromagnetic potentials,  $\hbar$ is the Planck constant, $m > 0$ is the mass of a particle and $q \in \bR$ is its charge.  For the sake of simplicity we suppose $\hbar = 1$ and $q = 1$ hereafter. \par
In the present paper we will consider scalar potentials $V(t,x)$ that are time-dependent and growing  polynomially in $\bR^d_x$.  That is,  
\begin{equation} \label{1.2}
C_0<x>^{2(M+1)} - C_1 
 \leq V(t,x) \leq C_2<x>^{2(M+1)}
 \end{equation}
in $\domain$ with constants $M  \geq 0, C_0 > 0, C_1 \geq 0$ and $C_2 > 0$, where $|x| = \left(\sumjd x_j^2\right)^{1/2}$ and $<x>  = \left(1 + |x|^2\right)^{1/2}$. We denote by $L^2 = L^2(\mathbb{R}^d)$  the space of all square integrable functions on
$\mathbb{R}^d$ with inner
product $(f,g) := \int f(x) g(x)^*dx$ and  norm $\Vert f\Vert$,  where $g^*$ is the complex conjugate of $g$. For a multi-index
$\alpha = (\alpha_1,\dotsc,\alpha_d)$  we  write $|\alpha| = 
\sum_{j=1}^d
\alpha_j$, $\partial_{x_j} = \partial /\partial
x_j$ and  $\partial_x^{\alpha} = \partial_{x_1}^{\alpha_1}
\cdots \partial_{x_d}^{\alpha_d}$. 
For $M \geq 0$ in \eqref{1.2} let us introduce the weighted Sobolev spaces 
\begin{align} \label{1.3}
B^a(\mathbb{R}^d)  := \{f \in & L^2(\mathbb{R}^d);
 \|f\|_a := \|f\| + \notag\\
& \sum_{|\alpha| \leq  2a} \|\partial_x^{\alpha}f\| +
\|<\cdot>^{2a(M+1)}f\| < \infty\}
 \end{align}
for $a = 1,2,\dotsc$.  We denote the dual space of $B^a\ (a = 1,2,\dotsc)$ by $B^{-a}$ and the $L^2$ space by $B^0$.  \par
  The main aim in the present paper is to prove that for any $u_0 \in B^a\ (a =0, \pm1,\pm2,\dots)$ there exists the unique solution $u(t) \in \mathcal{E}^0_t([0,T];B^a) \cap \mathcal{E}^1_t([0,T]; \\
  B^{a-1}) $ with $u_0$ at $t = 0$ to \eqref{1.1}, where ${\cal E}^j_t([0,T];B^a)\ (j = 0,1,\dotsc)$ denotes the space of all $B^a$-valued j-times continuously differentiable functions on $[0,T]$. 
  \par
  Our results above will be applied in \cite{Ichinose 2017} to the proof of the convergence of the Feynman path integrals for the Schr\"odinger equations \eqref{1.1} with potentials growing polynomially in $\bR^d_x$, i.e. satisfying \eqref{1.2}.  The proof of the convergence of the Feynman path integrals for such Sch\"odinger equations has been hardly obtained (cf. \S 10.2 in \cite{Albeverio et all}).  It should be noted that  the existence and the uniqueness of  solutions  not only in $L^2$ but also in $B^1$ have been necessary to prove the convergence in the $L^2$ space of the Feynman path integrals as seen in \cite{Ichinose 1999, Ichinose 2014}.   \par
  The results of the existence and the uniqueness of solutions to \eqref{1.1} will be extended for  the multi-particle systems.  For the sake of simplicity we will consider  the 4-particle systems
\begin{align} \label{1.4}
 i\frac{\partial u}{\partial t}(t) & = H(t)u(t)
 := \Biggl[\sum_{k=1}^4 \Biggl\{\frac{1}{2m_k}
\sum_{j=1}^d
      \left(\frac{1}{i}\frac{\partial}{\partial x_j^{(k)}} - A_j^{(k)}(t,x^{(k)})\right)^2 + 
      \notag\\  
   &   V_k(t,x^{(k)})\Biggr\}
     + \sum_{1\leq i <j \leq 4}W_{ij}(t,x^{(i)}-x^{(j)})\Biggr]u(t) \notag \\
     & \equiv \Biggl[\sum_{k=1}^4 H_k(t) + \sum_{1\leq i <j \leq 4}W_{ij}(t,x^{(i)}-x^{(j)})\Biggr]u(t),
\end{align}
where $x^{(k)} \in \bR^d\ (k = 1,2,3,4).$   Our results in the present paper for the 4-particle systems will be easily extended for the general multi-particle systems.
\par
	In addition,  we suppose that the potentials of \eqref{1.1} and \eqref{1.4}  are depending on  a parameter. The second aim in the present paper is to prove that if potentials depend continuously and differentiably on the parameter, the solutions to \eqref{1.1} and \eqref{1.4} respectively become  continuous and  differentiable with respect to its parameter   in ${\cal E}^0_t([0,T];B^a)\ (a = 0,\pm1,\pm2,\dots)$.  Such  results have been well known in the theory of ordinary differential equations as the fundamental ones.	\par
	When the Hamiltonian $H(t)$ is independent of $t \in [0,T]$, i.e. $H(t) = H$,  the existence and the uniqueness of solutions in the $L^2$ space to \eqref{1.1} and \eqref{1.4}  are equivalent to the self-adjointness of $H$ (cf. \S 8.4 in \cite{Reed-Simon I}).  The self-adjointness of $H$ in $L^2$ has almost been settled now (cf. \cite{Cycon et all, Leinfelder et all, Reed-Simon II}), as stated in the introductions of \cite{Yajima 1991,Yajima 2011,Yajima 2016}. 
	\par
	It should also be noted that the Hamiltonian $H_0 = - \sumjd \partial_{x_j}^2 - a|x|^b$  with constants $a > 0$ and $b > 2$ on $\Cinfty_0(\bR^d)$ is not essentially self-adjoint in $L^2$ (cf. pp. 157-159 in \cite{Berezin et all}), $H_0$ has equal deficiency (cf. Theorem X.3 in \cite{Reed-Simon II}) and so $H_0$ has an infinite number of  different self-adjoint extensions in $L^2$ from  Theorem X.2 and its corollary in \cite{Reed-Simon II}, where  $\Cinfty_0(\bR^d)$ denotes the space of all infinitely differentiable functions with compact support on $\bR^d$.  Consequently, as well known (cf. Theorem VIII.7 in \cite{Reed-Simon I}), the Stone theorem shows that the uniqueness of solutions to \eqref{1.1} with $H(t) = H_0$ doesn't always hold in $L^2$. \par
	If $H(t)$ is not independent of $t \in [0,T]$, the problem is never simple.  In \cite{Yajima 1991} Yajima has proved the existence and the uniqueness of solutions to \eqref{1.1} in $B^a\ (a = 0,\pm1,\pm2,\dots)$ with $M = 1$ under the assumptions
	\begin{align} \label{1.5}
	&  |\partial^{\alpha}_xV(t,x)| \leq C_{\alpha}, \ |\alpha| \geq 2, \notag\\
	& \sumjd \bigl(|\partial^{\alpha}_xA_j(t,x)| + |\partial^{\alpha}_x\partial_tA_j(t,x)|\bigr)
	 \leq C_{\alpha},\ |\alpha| \geq 1, \notag\\
	 & \sum_{1 \leq j < k \leq d}|\partial^{\alpha}_xB_{jk}(t,x)| \leq C_{\alpha}<x>^{-(1+\delta_{\alpha})},\ |\alpha| \geq 1
	\end{align}
	with constants $\delta_{\alpha} > 0$ and $C_{\alpha} \geq 0$ by  the theory of Fourier integral operators, where  $B_{jk} =  \partial A_k/\partial x_j  -\partial A_j/\partial x_k. $  In the present paper we often use symbols $C, C_{\alpha}, C_{\alpha,\beta}$, $C_a, \delta$ and $\delta_{\alpha}$ to write down constants, though these values are different in general. 
 In \cite{Ichinose 1995} the author has proved the existence and the uniqueness of solutions in $B^a\ (a = 0,\pm1,\pm2,\dots)$ with $M = 1$ under the assumptions 
	\begin{equation} \label{1.6}
	  |\partial^{\alpha}_xV(t,x)| \leq C_{\alpha}<x>,\ |\alpha| \geq 1,
	\end{equation}
	\begin{equation} \label{1.7}
 \sumjd |\partial^{\alpha}_xA_j(t,x)|
	 \leq C_{\alpha},\ |\alpha| \geq 1\end{equation}
	by the energy method.  Recently, general results about the existence and the uniqueness of solutions to \eqref{1.1} in the $L^2$ space have been obtained in \cite{Yajima 2011} by the semi-group method.  In \cite{Yajima 2016}  results  in $B^a\ (a = 0,1)$ with $M=1$ for  the multi-particle systems, e.g. \eqref{1.4}  have been obtained with singular potentials $W_{ij}$ under the assumptions \eqref{1.5} for $(V_k,A^{(k)})$ by the theories of semi-groups and Fourier integral operators.\par
	It should be noted that our results in the present paper and in addition even the results in \cite{Yajima 2011} are not enough to study the equations \eqref{1.1} in a general way.  For example, both of these results can not be applied to the simple equations \eqref{1.1} with $V = a(t)|x|^4 + |x|^2$ and $A = 0$ where $a(0) = 0$ and $a(t) > 0 \ (t \in (0,T])$, as mentioned in Remark 2.1 of the present paper. \par
	Next let us consider the Schr\"odinger equations with potentials dependent on a parameter.  When the Hamiltonian $H(t)$ is independent of $t \in [0,T]$, it follows from Theorems VIII. 21 and VIII. 25 in \cite{Reed-Simon I} that if potentials are continuous with respect to its parameter, in the $L^2$ space so are the solutions to \eqref{1.1} and \eqref{1.4}.  \par
	If  $H(t)$ is not independent of $t \in [0,T]$, the problem is never simple like the existence and the uniqueness of solutions.  In \cite{Ichinose 2012} the author has proved that if  potentials  depend continuously and differentiably on a parameter under the assumptions \eqref{1.6} and \eqref{1.7}, the solutions to \eqref{1.1} respectively become  continuous and  differentiable  with respect to its parameter in $\mathcal{E}_t^0([0,T];B^a)\ (a = 0,1,2,\dots)$ with $M =1$. \par
	All proofs of our results will be given by the energy method as in \cite{Ichinose 1995} and \cite{Ichinose 2012}.  The crucial point in the proofs of our results for \eqref{1.1} is to introduce a family of bounded operators $\big\{H_{\epsilon}(t)\bigr\}_{0 <\epsilon \leq 1}$ in $B^a\ (a = 0,\pm1,\pm2,\dots)$ by (4.1) as an approximation of $H(t)$ in \eqref{1.1}.  Then we can prove Proposition 4.2, by which we can complete the proofs of our results for \eqref{1.1} as in \cite{Ichinose 1995} and \cite{Ichinose 2012}.  In the same way the crucial point in the proofs of our results for \eqref{1.4} is to introduce
 $\big\{H_{\epsilon}(t)\bigr\}_{0 <\epsilon \leq 1}$  by \eqref{5.29}	as an approximation of $H(t)$ in \eqref{1.4}.  As in the proofs for \eqref{1.1} we can complete the proofs for \eqref{1.4}. 
 We note that the results in the present paper for the 4-particle systems \eqref{1.4} give generalizations of those for \eqref{1.1} in the present paper and  \cite{Ichinose 1995,Ichinose 2012}.
 \par
 	The plan of the present paper is as follows.  In \S 2 we will state all theorems.  \S 3 is devoted to preparing for the proofs of the theorems for \eqref{1.1}.  In \S 4 we will prove all theorems for \eqref{1.1}.  In \S 5  we will prove all theorems for \eqref{1.4}.
\section{Theorems}
{\bf Assumption 2.1. } We assume for all $\alpha$ and $k = 0,1$ that $\partial_x^{\alpha}\partial_t^kV(t,x)$ and $\partial_x^{\alpha}\partial_t^kA_j(t,x)\ (j = 1,2,\dots,d)$ are continuous in $\domain$. In $\domain$ we assume  \eqref{1.2} with 
constants $M \geq 0, C_0 > 0, C_1 \geq 0$ and $C_2 > 0$, and the following for $\jdots$.  We have
\begin{equation} \label{2.1}
|\partial_x^{\alpha}V(t,x)| \leq C_{\alpha}<x>^{2(M+1)}, |\alpha| \geq 1,
\end{equation}
\begin{equation} \label{2.2}
 |\partial_x^{\alpha}\partial_tV(t,x)| \leq C_{\alpha}<x>^{2(M+1)}\end{equation}
for all $\alpha$,
\begin{equation} \label{2.3}
 |A_j(t,x)| \leq C<x>^{M+1-\delta}
 \end{equation}
with a constant $\delta > 0$,
\begin{equation} \label{2.4}
|\partial_x^{\alpha}A_j(t,x)| \leq C_{\alpha}<x>^{M+1}, |\alpha| \geq 1
 \end{equation}
and 
\begin{equation} \label{2.5}
 |\partial_x^{\alpha}\partial_tA_j(t,x)| \leq C_{\alpha}<x>^{M+1}
 \end{equation}
for all $\alpha$.  \vspace{0.3cm}\par
	Let $B^a$ be the weighted Sobolev spaces introduced in \S 1.
	\begin{thm}
	Under Assumption 2.1 for any $u_0 \in B^a\ (\adots)$ there exists the unique solution $u(t) \in \mathcal{E}_t^0([0,T];B^a) \cap \mathcal{E}_t^1([0,T];B^{a-1})$ with $u(0) = u_0$ to \eqref{1.1}.  This solution $u(t)$ satisfies 
\begin{equation} \label{2.6}
 \Vert u(t) \Vert_a \leq C_a \Vert u_0 \Vert_a\quad (0 \leq t \leq T)
 \end{equation}
and in particular
\begin{equation} \label{2.7}
 \Vert u(t) \Vert =  \Vert u_0 \Vert \quad (0 \leq t \leq T).
 \end{equation}
	\end{thm}
	\begin{rem}
	Let $a(t)$ be a continuous function on $[0,T]$ such that $a(0) = 0$ and $a(t) > 0\ (0 < t \leq T)$.  Since $V := a(t)|x|^4 + |x|^2$ does not satisfy  \eqref{1.2}  for any $M \geq 0$, Theorem 2.1 can not be applied to \eqref{1.1} with $H(t) := (1/2m)\sumjd (-i\partial_{x_j})^2 +a(t)|x|^4 + |x|^2$.  In addition, Theorems 1.2 and 1.4 in \cite{Yajima 2011} can not be applied either, because these self-adjoint operators $H(t)\ (0 \leq t \leq T)$ in $L^2(\bR^d)$ don't have a common domain.
	\end{rem}
	Next, let us consider the Schr\"odinger equations \eqref{1.1} with potentials $(V(t,x;\rho), \\
	A(t,x;\rho))$  dependent on a parameter $\rho \in \mathcal{O}$, where $\mathcal{O}$ is an open set in $\bR$.
	\begin{thm}
	We suppose that $(\Vtxrho,\Atxrho)$ for all $\rho \in \mathcal{O}$ satisfy Assumption 2.1 and have the uniform estimates \eqref{1.2} and \eqref{2.1} - \eqref{2.5}  with respect to $\rho \in \mathcal{O}$. 
	In addition, we assume that $\partial_x^{\alpha}\Vtxrho$ and  $\partial_x^{\alpha}A_j(t,x;\rho)\ (\jdots)$  for all $\alpha$ are continuous in $[0,T]\times \bR^d \times \mathcal{O} $. 
	 Let $u_0 \in B^a \ (\adots)$ be independent of $\rho$ and $u(t;\rho)$
	the solutions to \eqref{1.1} with $u(0;\rho) = u_0$ determined in Theorem 2.1.  Then, the mapping : $\mathcal{O} \ni
\rho	\rightarrow \utrho \in \Eta$ is continuous, where the norm in $\Eta$ is $\max_{0 \leq t \leq T}\Vert f(t)\Vert_a.$
	\end{thm}

	We set
\begin{equation} \label{2.8}
 h(t,x,\xi) :=  \frac{1}{2m}|\xi - A(t,x)|^2 + V(t,x).
 \end{equation}
Let $\Sspace$ denote the Schwartz space of all   rapidly decreasing functions on $\bR^d$ and $\chi \in \Sspace$ such that $\chi(0) = 1$.
Then, using the oscillatory integral, we can write $H(t)f$ in \eqref{1.1} for $f \in \Sspace$ as
\begin{align} \label{2.9}
& H(t)f = H\left(t,\frac{X+X'}{2},D_x\right)f  = \text{Os}-\iint e^{i(x-y)\cdot\xi}h\left(t,\frac{x+y}{2},\xi\right)f(y)dy\hspace{0.08cm}\dbar\xi \notag \\
& := \lim_{\epsilon\rightarrow 0} \iint e^{i(x-y)\cdot\xi}\chi(\epsilon \xi)h\left(t,\frac{x+y}{2},\xi\right)f(y)dy\hspace{0.08cm}\dbar\xi, \quad  \dbar\xi = (2\pi)^{-d}d\xi
 \end{align}
in the form of the pseudo-differential operator with the double symbol $h(t,(x+x')/2,\xi)$ (cf. \cite{Kumano-go}).
\begin{thm}
Besides the assumptions of Theorem 2.2 we suppose for all $\alpha$ and $\jdots$ that $\partial_{\rho}\partial^{\alpha}_x\Vtxrho$ and $\partial_{\rho}\partial^{\alpha}_xA_j(t,x;\rho)$ are continuous in $\domain \times \mathcal{O}$ and satisfy
\begin{equation} \label{2.10}
\sup_{\rho \in \mathcal{O}} |\partial_{\rho}\partial^{\alpha}_x\Vtxrho| \leq C_{\alpha}<x>^{2(M+1)}
\end{equation}
\begin{equation} \label{2.11}
 \sup_{\rho \in \mathcal{O}} |\partial_{\rho}\partial^{\alpha}_xA_j(t,x;\rho)| \leq C_{\alpha}<x>^{M+1}
 \end{equation}
in $\domain \times \mathcal{O}$.  Let $u_0 \in B^{a+1}\ (\adots)$ be independent of $\rho$ and $\utrho$ the solutions to \eqref{1.1} with $u(0) = u_0$.  Then, the mapping : $ \mathcal{O} \ni \rho \rightarrow \utrho \in \Eta$ is continuously differentiable with respect to $\rho$, we have
\begin{equation} \label{2.12}
\sup_{\rho \in \mathcal{O}} \Vert \partial_{\rho}\utrho\Vert_a \leq C_a \Vert u_0 \Vert_{a+1}\quad (0 \leq t \leq T)
 \end{equation}
and $\partial_{\rho}\utrho$ is the solution to
\begin{equation} \label{2.13}
 i\frac{\partial}{\partial t}\wtrho = H(t;\rho)\wtrho + \frac{\partial H(t;\rho)}{\partial\rho}\utrho
  \end{equation}
with $w(0) = 0$, where $\partial_{\rho}H(t;\rho)$ is the pseudo-differential operator with the double symbol $\partial_{\rho}h(t,(x+x')/2,\xi;\rho)$.
\end{thm}
	Now, we consider the 4-particle systems \eqref{1.4}.\par
	{\bf Assumption 2.2.}  We assume the following in $\domain$ : (1) Each $(V_k(t,x),A^{(k)}(t,x))\ (k = 1,2)$ satisfies Assumption 2.1 with $M = M_k \geq 0$. (2)  Each $(V_k,A^{(k)})\ (k = 3,4)$ satisfies \eqref{1.6} and \eqref{1.7}.
	(3)  For $M_0 := \min(M_1,M_2)$   $W_{12}$ satisfies
\begin{equation} \label{2.14}
|W_{12}(t,x)| \leq C<x>^{2(M_0 +1)-\delta}
  \end{equation}
with a constant $\delta > 0$ and 
\begin{equation} \label{2.15}
|\partial_x^{\alpha}W_{12}(t,x)| \leq C_{\alpha}<x>^{2(M_0 +1)},\  |\alpha| \geq 1.
  \end{equation}
(4) Each $W_{ij}(t,x)$ except $W_{12}$ satisfies \eqref{1.6}. \par
	We introduce the weighted Sobolev spaces $
B'^a(\mathbb{R}^{4d})  := \{f \in  L^2(\mathbb{R}^{4d});
 \|f\|_a := \|f\| + 
 \sum_{|\alpha| \leq  2a} \|\partial_x^{\alpha}f\| + \sum_{k=1}^4
\|<x^{(k)}>^{2a(M_k+1)}f\| < \infty\}\ (a = 1,2,\dots)
$ with $M_3 = M_4 = 0$, and denote the dual space of $B'^a$ by $B'^{-a}$ and $L^2$ by $B'^0$. 
\begin{thm}
Under Assumption 2.2 for any $u_0 \in B'^a(\bR^{4d})\ (a = 0,\pm1,\pm2,\\
\dots)$ there exists the unique solution $u(t) \in \mathcal{E}_t^0([0,T];B'^a) \cap \mathcal{E}_t^1([0,T];B'^{a-1})$ with $u(0) = u_0$ to \eqref{1.4}.  This solution $u(t)$ satisfies 
\begin{equation} \label{2.16}
 \Vert u(t) \Vert_a \leq C_a \Vert u_0 \Vert_a\quad (0 \leq t \leq T)
 \end{equation}
and in particular
\begin{equation} \label{2.17}
 \Vert u(t) \Vert =  \Vert u_0 \Vert \quad (0 \leq t \leq T).
 \end{equation}
\end{thm}
	We will consider the 4-particle systems \eqref{1.4} with potentials dependent on a parameter $\rho \in \mathcal{O}$.
\begin{thm}
We suppose that $(\Vktxrho,\Aktxrho)\ (k= 1,2,3,4)$ and $\Wijtxrho\ (1 \leq i < j \leq 4)$ for all $\rho \in \mathcal{O}$ satisfy Assumption 2.2 and have the uniform estimates   with respect to $\rho \in \mathcal{O}$ \eqref{1.2} and \eqref{2.1}-\eqref{2.5} for $(V_k, A^{(k)})\ (k = 1,2)$ with $M = M_k$, \eqref{1.6}-\eqref{1.7} for $(V_k, A^{(k)})\ (k = 3,4)$,  \eqref{2.14}-\eqref{2.15} for $W_{12}$ and \eqref{1.6} for $W_{ij}$ except $W_{12}$.  In addition, we assume that $\partial_x^{\alpha}V_k(t,x;\rho)$,  $\partial_x^{\alpha}A^{(k)}_j(t,x;\rho)\ (k = 1,2,3,4, \jdots)$ and  $\partial_x^{\alpha}W_{ij}(t,x;\rho)\ (1 \leq i < j \leq 4)$  for all $\alpha$ are continuous in $[0,T]\times \bR^d \times \mathcal{O}$.
\par
Let $u_0 \in B'^a \ (\adots)$ be independent of $\rho$ and $u(t;\rho)$
	the solutions to \eqref{1.4} with $u(0;\rho) = u_0$ determined in Theorem 2.4.  Then, the mapping : $\mathcal{O} \ni
\rho	\rightarrow \utrho \in \Etaprime$ is continuous. 
\end{thm}
\begin{thm}
Besides the assumptions of Theorem 2.5 we suppose for all $\alpha$ that all functions $\partial_{\rho}\partial^{\alpha}_x\Vktxrho, \partial_{\rho}\partial^{\alpha}_xA^{(k)}_j(t,x;\rho)$ and $\partial_{\rho}\partial^{\alpha}_xW_{ij}(t,x;\rho)$ are continuous in $\domain \times \mathcal{O}$.  In addition, we assume \eqref{2.10}-\eqref{2.11} for $(V_k,A^{(k)})\ (k = 1,2)$ with $M = M_k$,
\begin{equation} \label{2.18}
\sup_{\rho \in \mathcal{O}} |\partial_{\rho}\partial^{\alpha}_x\Vktxrho| \leq C_{\alpha}<x>,\ |\alpha| \geq 1,
\end{equation}
\begin{equation} \label{2.19}
\sup_{\rho \in \mathcal{O}} |\partial_{\rho}\partial^{\alpha}_x\Aktxrho| \leq C_{\alpha},\ |\alpha| \geq 1
\end{equation}
in $\domain$ for $k = 3,4$, 
\begin{equation} \label{2.20}
\sup_{\rho \in \mathcal{O}} |\partial_{\rho}\partial^{\alpha}_xW_{12}(t,x;\rho)| \leq C_{\alpha}<x>^{2(M_0+1)}
\end{equation}
for all $\alpha$ and 
\begin{equation} \label{2.21}
\sup_{\rho \in \mathcal{O}} |\partial_{\rho}\partial^{\alpha}_xW_{ij}(t,x;\rho)| \leq C_{\alpha}<x>, \ |\alpha| \geq 1
\end{equation}
for $(i,j) \not= (1,2).$
\par
 Let $u_0 \in B'^{a+1}\ (\adots)$ be independent of $\rho$ and $\utrho$ the solutions to \eqref{1.4} with $u(0;\rho) = u_0$.  Then we have the same assertion as in Theorem 2.3.
 \end{thm}
 \begin{rem}
 Theorems 2.4 - 2.6 in the present paper give generalizations of Theorems 2.1 - 2.3 in the present paper, Theorem in \cite{Ichinose 1995} and Theorems 2.1 - 2.4 in \cite{Ichinose 2012}.
 \end{rem}
 \section{Preliminaries}
 Let $h(t,x,\xi)$ be the function defined by \eqref{2.8}.
 \begin{lem}
 Assume \eqref{1.2} and \eqref{2.3}.  Then, there exist constant $C_0^* > 0$ and $C_1^*\geq 0$ such that 
 \begin{equation} \label{3.1}
 C_0^*(<\xi>^2 +  <x>^{2(M+1)}) - C_1^* \leq h(t,x,\xi)  \leq  C_0^{*-1}(<\xi>^2 +  <x>^{2(M+1)})
 \end{equation}
 in $\domain$.
 \end{lem}
 \begin{proof}
 From \eqref{2.8} we have $h(t,x,\xi) \leq (|\xi|^2 + |A(t,x)|^2)/m + V(t,x)$ and hence by \eqref{1.2} and \eqref{2.3}
 \begin{equation*}
 h(t,x,\xi) \leq C(<\xi>^2 + <x>^{2(M+1)})
 \end{equation*}
 in $\domain$ with a constant $C \geq 0$. \par
   We may assume $0 < \delta \leq M+1$ in \eqref{2.3}.  Take $p > 1$ and $q > 1$ so that 
 \begin{equation*}
\frac{1}{p} = \frac{1}{2}\left(1 - \frac{1}{2}\cdot\frac{\delta}{M+1}\right),\quad \frac{1}{p} + \frac{1}{q} = 1.
 \end{equation*}
Then we have
 \begin{align*}
& p(M+1 - \delta) = 2(M+1)\cdot\frac{1-\frac{\delta}{M+1}}{1-\frac{1}{2}\cdot\frac{\delta}{M+1}} \equiv 2(M+1)\delta_1,
\notag \\
&q = \frac{2}{1+\frac{1}{2}\cdot\frac{\delta}{M+1}} \equiv 2\delta_2
 \end{align*}
 with $0 < \delta_j < 1\ (j =  1,2).$  Hence, Young's inequality and \eqref{2.3} show
 \begin{align*}
&  |A(t,x)|\cdot|\xi| \leq \frac{1}{p}|A|^p + \frac{1}{q}|\xi|^q \leq \frac{1}{p}<x>^{p(M+1-\delta)} + \frac{1}{q}|\xi|^q
\notag \\
&= \frac{1}{p}<x>^{2(M+1)\delta_1} + \frac{1}{q}|\xi|^{2\delta_2}.
\end{align*}
 Applying this, \eqref{1.2} and \eqref{2.3} to \eqref{2.8}, we have
 \begin{align*} 
  h(t,x,\xi)  & \geq  \frac{1}{2m}\left(|\xi|^2  - 2|A| \cdot |\xi|\right) + V \\
   & \geq C_0(<\xi>^2 - <x>^{2(M+1)\delta_1} -<\xi>^{2\delta_2} + <x>^{2(M+1)}) - C_1
 \end{align*}
 with constants $C_0 > 0$ and $C_1 \geq 0$.  Therefore, we obtain \eqref{3.1}.
 \end{proof}
 	We fix $C_0^*$ and $C_1^*$ in Lemma 3.1 hereafter.  We set 
	\begin{equation} \label{3.2}
	h_s(t,x,\xi) = h(t,x,\xi) + \frac{i}{2m}\nabla\cdot A(t,x),
	\end{equation}
	where $\nabla\cdot A(t,x) = \sum_{j=1}^d\partial_{x_j}A_j(t,x).$
	Since  the real part $\rittaire h_s(t,x,\xi)$ of $h_s(t,x,\xi)$ is equal to $h(t,x,\xi)$, we can determine
	\begin{equation} \label{3.3}
	p_{\mu}(t,x,\xi) := \frac{1}{\mu + h_s(t,x,\xi)}
	\end{equation}
	for $\mu \geq C_1^*$ under the assumptions of Lemma 3.1.  We  denote by $H_s(t,X,D_x)f$ the pseudo-differential operator 
	\begin{equation*} 
	\int e^{ix\cdot\xi}h_s(t,x,\xi)\hdbar\xi\int e^{-iy\cdot\xi}f(y)dy
	\end{equation*}
	for $f \in \Sspace$ with the symbol $h_s(t,x,\xi)$.  As is well known (cf. Theorem 2.5 in Chapter 2 of \cite{Kumano-go}), $H_s(t,X,D_x) = H(t)$ holds, where $H(t)$ is the operator defined by \eqref{1.1} or \eqref{2.9}.
	\begin{lem}
	Assume \eqref{1.2}, \eqref{2.1} and \eqref{2.3} - \eqref{2.4}.  Then we have
	\begin{equation} \label{3.4}
	\bigl[\mu + H(t)\bigr]P_{\mu}\txdx = I + R_{\mu}\txdx,
	\end{equation}
	\begin{equation} \label{3.5}
	\left|r^{(\alpha)}_{\mu\ (\beta)}(t,x,\xi)\right| \leq  C_{\alpha\beta}\left(\mu - C_1^*\right)^{-1/2}
	\end{equation}
	in $\domains$ for $\mu \geq C_0^*/2 + C_1^*$ with constants $C_{\alpha\beta}$ independent of $\mu$, where $r^{(\alpha)}_{\mu\ (\beta)} = \partial_{\xi}^{\alpha}(-i\partial_x)^{\beta}r_{\mu}$. 
	\end{lem}
	\begin{proof}
	Let $\mu \geq C_1^*$.
	By Lemma 3.1 and \eqref{3.2} we have
 \begin{equation} \label{3.6}
  C_0^*(<\xi>^2  + <x>^{2(M+1)}) + \mu - C_1^* \leq   \mu + \rittaire h_s(t,x,\xi).
   \end{equation}
     Since $H(t) = H_s\txdx$, from (2.13) in \cite{Ichinose 1995} we have
   \begin{align} \label{3.7}
  & r_{\mu}(t,x,\xi) = \sum_{|\alpha| = 1} \int_0^1 d\theta\  \text{Os}-\iint e^{-iy\cdot\eta}h_s^{(\alpha)}(t,x,\xi+\theta\eta)
  p_{\mu (\alpha)}(t,x+y,\xi)dy\hdbar\eta  \notag\\
 &  =  \sum_{|\alpha| = 1} \int_0^1 d\theta\  \text{Os}-\iint e^{-iy\cdot\eta}
 <y>^{-2l_0}  <D_{\eta}>^{2l_0}<\eta>^{-2l_1} <D_y>^{2l_1} \notag\\
 & \hspace{2cm} \cdot h_s^{(\alpha)}(t,x,\xi+\theta\eta)  p_{\mu (\alpha)}(t,x+y,\xi)dy\hdbar\eta
 \end{align}
  for large integers $l_0$ and $l_1$, where $<D_\eta>^2 = 1 - \sumjd \partial_{\eta_j}^2.$ \par
  	Now, using \eqref{2.1} and \eqref{2.3} - \eqref{2.4}, from \eqref{3.2} we have
   \begin{align*} 
  & |\partial_{x_j} h_s(t,x+y,\xi)| \leq C\left(<\xi><x+y>^{M+1} + <x + y>^{2(M+1)}\right) \notag\\
 &   \leq C'\left(<\xi>^2 + <x + y>^{2(M+1)}\right).
 \end{align*}
 In the same way we can prove
 \begin{equation} \label{3.8}
 |h^{(\alpha)}_{s\ (\beta)}(t,x+y,\xi)| \leq  C \left(<\xi>^2 + <x+y >^{2(M+1)}\right) 
 \end{equation}
 for all $\alpha$ and $|\beta|\geq 1$, and 
 \begin{equation} \label{3.9}
 |h^{(\alpha)}_{s\ (\beta)}(t,x,\xi + \theta\eta)| \leq  C \left(<\xi>+ <x >^{M+1}\right) <\eta>
 \end{equation}
 for $|\alpha|\geq 1$ and all $\beta$.  We also note
   \begin{align*} 
  & \frac{1}{<\xi>^2 + <x+y>^{2(M+1)}} \leq  \frac{1}{<\xi>^2 + <x>^{2(M+1)}/(\sqrt{2}<y>)^{2(M+1)}}\notag\\
 &   \leq \frac{(\sqrt{2}<y>)^{2(M+1)}}{<\xi>^2 + <x>^{2(M+1)}}.
 \end{align*}
  Apply \eqref{3.6} and \eqref{3.8} - \eqref{3.9} to \eqref{3.7}.  Then, 
  taking integers $l_0$ and $l_1$ so that $2l_0- 2(M+1) > d$ and $2l_1 - 1 > d$, we have 
   \begin{align} \label{3.10}
  & |r_{\mu}(t,x,\xi)| \leq C \iint
 <y>^{-2l_0} <\eta>^{-2l_1}<\eta> <y>^{2(M+1)}dy\hdbar\eta \notag\\
 & \hspace{1cm} \times \frac{\Theta^{1/2}}{\mu - C_1^* + C_0^*\Theta}	\leq C' \max_{1 \leq \theta}  \frac{\theta^{1/2}}{\mu - C_1^* + C_0^*\theta}
 \end{align}
 with constants $C$ and $C'$ independent of $\mu \geq C_1^*$, where  $\Theta = <\xi>^2 + \\
 <x>^{2(M+1)}$.
 Applying (2.9) in \cite{Ichinose 1995} with $\kappa = 1$ and $\tau = 2$ to \eqref{3.10}, we have
 \begin{equation*} 
  |r_{\mu}(t,x,\xi)| \leq C_0(\mu - C_1^*)^{-1/2} 
  \end{equation*}
 for $\mu \geq C_0^*/2 + C_1^*.$  In the same way we can prove \eqref{3.5} from \eqref{3.7} - \eqref{3.9}.
	\end{proof}
	\begin{pro}
	Under the assumptions of Lemma 3.2 there exist a constant $\mu \geq C_0^*/2 + C_1^*$ and a function $w(t,x,\xi)$ in $\domains$ satisfying 
 \begin{equation} \label{3.11}
  |w^{(\alpha)}_{\ \  (\beta)}(t,x,\xi)| \leq C_{\alpha\beta}\left(<\xi>^2 + <x>^{2(M+1)}\right)^{-1}
  \end{equation}
  for all $\alpha, \beta$ and 
 \begin{equation} \label{3.12}
 W\txdx = \bigl(\mu+ H(t)\bigr)^{-1}.
   \end{equation}
	\end{pro}
	\begin{proof} Let $\mu \geq C_0^*/2 + C_1^*$.
	From \eqref{3.3}, \eqref{3.6} and \eqref{3.8} - \eqref{3.9} we see
 \begin{equation*} 
  |p^{(\alpha)}_{\mu\  (\beta)}(t,x,\xi)| \leq C_{\alpha\beta}\left(<\xi>^2 + <x>^{2(M+1)}\right)^{-1}
  \end{equation*}
    for all $\alpha$ and $ \beta$.  Therefore, we can complete the proof of Proposition 3.3 as in the proof of (2.16) of \cite{Ichinose 1995} by using Lemma 3.2.
	\end{proof}
	We take a constant $\mu \geq C_0^*/2 + C_1^*$ in Proposition 3.3 and fix it hereafter throughout \S 3 and \S4.  Set
 \begin{equation} \label{3.13}
\lambda(t,x,\xi) := \mu + h_s(t,x,\xi).
   \end{equation}
 Then, from \eqref{3.2} we have
 \begin{equation} \label{3.14}
\Lambda\txdx = \mu + H_s(t,X,D_x) = \mu + H(t).
   \end{equation}
We take a $\chi \in \Sspace$ such that $\chi(0) = 1$ and set 
 \begin{equation} \label{3.15}
\chi_{\epsilon}(t,x,\xi) := \chi\bigl(\epsilon(\mu + h(t,x,\xi)\bigr)
   \end{equation}
 for constants $0 < \epsilon \leq 1$.  We note that $h(t,x,\xi)$ defined by \eqref{2.8} is a real-valued function. \par
 	The following is crucial in the present paper.
	\begin{lem}
	Under the assumptions of Lemma 3.2 there exist functions $k_{\epsilon}(t,x,\xi)\ (\zeo)$ in $\domains$ satisfying
 \begin{equation} \label{3.16}
 \sup_{0 < \epsilon \leq 1}\sup_{t,x,\xi} |k^{(\alpha)}_{\epsilon\  (\beta)}(t,x,\xi)| \leq C_{\alpha\beta} < \infty
  \end{equation}
  for all $\alpha, \beta$ and 
 \begin{equation} \label{3.17}
K_{\epsilon}\txdx = \Bigl[X_{\epsilon}\txdx, \Lambda\txdx\Bigr],
   \end{equation}
where $[\cdot,\cdot]$ denotes the commutator of operators.
	\end{lem}
	\begin{proof}
	Apply Theorem 3.1 in Chapter 2 of \cite{Kumano-go} to the right-hand side of \eqref{3.17}.  Then we have
   \begin{align} \label{3.18}
  & k_{\epsilon}(t,x,\xi) = \sum_{|\alpha| = 1} \Bigl\{\chi_{\epsilon}^{(\alpha)}(t,x,\xi)\lambda_{(\alpha)}(t,x,\xi)  
  - \lambda^{(\alpha)}(t,x,\xi)\chi_{\epsilon(\alpha)}(t,x,\xi) \Bigr\}\notag\\
 &\quad   + 2 \sum_{|\gamma| = 2}\frac{1}{\gamma\, !}\int_0^1 (1 - \theta)d\theta\  \text{Os}-\iint e^{-iy\cdot\eta}
 \Bigl\{\chi_{\epsilon}^{(\gamma)}(t,x,\xi+\theta\eta)\lambda_{(\gamma)}(t,x+y,\xi)  \notag\\
 & \qquad \quad 
  - \lambda^{(\gamma)}(t,x,\xi+\theta\eta)\chi_{\epsilon(\gamma)}(t,x+y,\xi) \Bigr\}dy\hdbar\eta 
  \equiv I_{1\epsilon} + I_{2\epsilon}.
 \end{align}
 By \eqref{3.2}, \eqref{3.13} and \eqref{3.15} we can write 
   \begin{align} \label{3.19}
  & I_{1\epsilon}(t,x,\xi) = \epsilon\chi\,'(\epsilon(\mu +h))\sum_{|\alpha| = 1}\Bigl\{h^{(\alpha)}h_{s(\alpha)}
  - h^{(\alpha)}_sh_{(\alpha)} \Bigr\}\notag\\
 & =  \epsilon\chi\,'\bigl(\epsilon(\mu +h(t,x,\xi))\bigr)\sum_{|\alpha| = 1}\frac{i}{2m^2}(\xi_{\alpha} - A_{\alpha}(t,x))(-i\partial_x)^{\alpha}\nabla\cdot A(t,x).
 \end{align}
 Hence, using $\epsilon\chi\,'(\epsilon(\mu +h))
 = (\mu +h)^{-1}\bigl\{\epsilon(\mu +h)\chi \,'(\epsilon(\mu +h))\bigr\}$ and Lemma 3.1, from \eqref{2.3} - \eqref{2.4} we can prove
 $
\sup_{\zeo}\sup_{t,x,\xi}|I_{1\epsilon}| < \infty. 
$
 In the same way from \eqref{3.19} we can prove
 \begin{equation} \label{3.20}
 \sup_{0 < \epsilon \leq 1}\sup_{t,x,\xi} |I^{(\alpha)}_{1\epsilon\  (\beta)}(t,x,\xi)| \leq C_{\alpha\beta} < \infty
  \end{equation}
  for all $\alpha$ and $\beta$.
  \par
	Next we will consider $I_{2\epsilon}$. Let $|\gamma| = 2$.  Since from \eqref{3.15} we have 
 \begin{equation*} 
\partial_{\xi_j}\chi_{\epsilon}(t,x,\xi) = \frac{1}{m}\epsilon \chi \,'(\epsilon(\mu +h))(\xi_j - A_j)
  \end{equation*}
  and
 \begin{equation*} 
\epsilon(\xi_j - A_j)\partial_{\xi_k}\chi \,'(\epsilon(\mu +h)) =  \frac{1}{m}\epsilon^2(\xi_j - A_j)(\xi_k - A_k)\chi \,''(\epsilon(\mu +h)),
  \end{equation*}
  as in the proof of \eqref{3.20} we can easily prove
 \begin{equation} \label{3.21}
\sup_{\zeo}  |\chi^{(\alpha+\gamma)}_{\epsilon\ \ (\beta)}(t,x,\xi)| \leq C_{\alpha\beta}\left(<\xi>^2 + <x>^{2(M+1)}\right)^{-1}
  \end{equation}
  for all $\alpha$ and $\beta$.  In the same way we can also prove
 \begin{equation} \label{3.22}
\sup_{\zeo}  |\chi^{(\alpha)}_{\epsilon\ (\beta+\gamma)}(t,x,\xi)| \leq C_{\alpha\beta} < \infty
  \end{equation}
   for all $\alpha$ and $\beta$.  We also note from \eqref{3.13} that each of $\lambda^{(\gamma)} = h_s^{(\gamma)}$ for $|\gamma| = 2$ is equal to $0$ or $1/m$.
   Hence, applying \eqref{3.8} and \eqref{3.21} - \eqref{3.22} to $I_{2\epsilon}$ in \eqref{3.18}, as in the proof of \eqref{3.10} we have 
 $
\sup_{\zeo}\sup_{t,x,\xi}|I_{2\epsilon}| < \infty. 
$
In the same way we can prove 
 \begin{equation*} 
 \sup_{0 < \epsilon \leq 1}\sup_{t,x,\xi} |I^{(\alpha)}_{2\epsilon\  (\beta)}(t,x,\xi)| \leq C_{\alpha\beta} < \infty
  \end{equation*}
   for all $\alpha$ and $\beta$, which 
   completes the proof of Lemma 3.4 together with \eqref{3.20}.
	\end{proof}
	Let 
 \begin{equation} \label{3.23}
\lambda_{M}(x,\xi) := \mu\,' + \frac{1}{2m}|\xi|^2 + <x>^{2(M+1)},
  \end{equation}
  which is equal to $\lambda(t,x,\xi)$ defined by \eqref{3.13} with $V =  <x>^{2(M+1)}$ and $A = 0$.  Let $B^a\ (\adots)$ 
be the weighted Sobolev spaces introduced in \S 1.
	\begin{pro}
	(1)  There exist a constant $\mu\,' \geq 0$  and a function $w_M(x,\xi)$ in $\domains$ satisfying 
	\eqref{3.11}  
  for all $\alpha, \beta$ and 
 \begin{equation} \label{3.24}
 W_M(X,D_x) = \Lambda_M(X,D_x)^{-1}.
   \end{equation}
(2) We take a $\mu\,' \geq 0$ satisfying (1).  Let $f$ be in the dual space $\mathcal{S}\,'(\bR^d)$ of $\Sspace$.  Then, $B^a \ni f \ (\adots)$ is equivalent to $ (\Lambda_M)^a f \in L^2.$
	\end{pro}
	\begin{proof}
	The assertion (1) follows from Proposition 3.3.  The assertion (2) follows from Lemma 2.4  of \cite{Ichinose 1995} with $s = a, a = 2(M+1)$ and $b=2$.
	\end{proof}
	Let $\mu\,' \geq 0$ be a constant in Proposition 3.5 and fix it hereafter throughout the present paper.  Under the assumptions of Lemma 3.2 from \eqref{3.1}, \eqref{3.8} and \eqref{3.9} we have
 \begin{equation} \label{3.25}
|h^{(\alpha)}_{s\  (\beta)}(t,x,\xi)| \leq C_{\alpha\beta}\weight
  \end{equation}
  in $\domains$ for all $\alpha$ and $\beta$.
  %
	\section{Proofs of Theorems 2.1 - 2.3}
	Let $\lambda(t,x,\xi)$ and $\chi_{\epsilon}\txxi\ (\zeo)$ be the functions defined by \eqref{3.13} and \eqref{3.15}, respectively.  We define the approximation of $H(t)$ by the product of operators
 \begin{equation} \label{4.1}
H_{\epsilon}(t) := X_{\epsilon}\txdx^{\dagger}H(t)X_{\epsilon}\txdx,
   \end{equation}
 where $X_{\epsilon}\txdx^{\dagger}$ denotes the formally adjoint operator of $X_{\epsilon}\txdx$.
 \begin{lem}
 Under Assumption 2.1 there exist functions $\qepsilon\txxi\ (\zeo)$ satisfying
 \begin{equation} \label{4.2}
 \sup_{0 < \epsilon \leq 1}\sup_{t,x,\xi} |q^{(\alpha)}_{\epsilon\  (\beta)}(t,x,\xi)| \leq C_{\alpha\beta} < \infty
  \end{equation}
  for all $\alpha, \beta$ and
 \begin{align} \label{4.3}
Q_{\epsilon}\txdx = &\Bigl[\Lambda\txdx, H_{\epsilon}(t)\Bigr]\Lambda\txdx^{-1} \notag \\
&  + i\frac{\partial\Lambda}{\partial t}\txdx\Lambda\txdx^{-1}.
   \end{align}
 \end{lem}
 \begin{proof}
 We first note 
 \begin{align*}
 & \Bigl[\Lambda\txdx, H_{\epsilon}(t)\Bigr] = \Bigl[\Lambda(t), X_{\epsilon}(t)^{\dagger}\Bigr]H(t)X_{\epsilon}(t) \\
 &\quad  + X_{\epsilon}(t)^{\dagger}\Bigl[\Lambda(t), H(t)\Bigr]X_{\epsilon}(t) + 
 X_{\epsilon}(t)^{\dagger}H(t)\Bigl[\Lambda(t), X_{\epsilon}(t)\Bigr].
 \end{align*}
 Since $\Lambda(t) = \mu + H(t)$ from \eqref{3.14}, we have 
 $[\Lambda(t),H(t)] = 0$ and  $\Lambda(t)^{\dagger} = \Lambda(t)$. Hence
 \begin{align} \label{4.4}
 & \Bigl[\Lambda(t), H_{\epsilon}(t)\Bigr] = -\Bigl[\Lambda(t), X_{\epsilon}(t)\Bigr]^{\dagger}H(t)X_{\epsilon}(t) +  X_{\epsilon}(t)^{\dagger}H(t)\Bigl[\Lambda(t), X_{\epsilon}(t)\Bigr].
   \end{align}
  Thereby, applying  Lemma 3.4 to \eqref{4.4}, we see from Proposition 3.3 and \eqref{3.25} that there exist functions $q_{1\epsilon}\txxi\ (\zeo)$ satisfying \eqref{4.2} and 
 \begin{align*} 
Q_{1\epsilon}\txdx = &\Bigl[\Lambda\txdx, H_{\epsilon}(t)\Bigr]\Lambda\txdx^{-1}.
 \end{align*}
 It is easy to study the second term in the right-hand side of  \eqref{4.3} by Proposition 3.3. Thus, our proof is complete.
 \end{proof}
 \begin{pro}
 Under Assumption 2.1 there exist functions $q_{a\epsilon}\txxi\ (\adots, \zeo)$ satisfying \eqref{4.2} and 
 \begin{equation} \label{4.5}
Q_{a\epsilon}\txdx = \Biggl[i\frac{\partial}{\partial t}- H_{\epsilon}(t), \Lambda(t)^a\Biggr]\Lambda(t)^{-a}.
   \end{equation}
 \end{pro}
 \begin{proof}
 For $a = 0$ the assertion is clear.  For $a = 1$ the assertion follows from Lemma 4.1. Let us consider the case $a = 2$.  We note
 \begin{equation} \label{4.6}
[P,QR] = [P,Q]R + Q[P,R]
   \end{equation}
and thereby
 \begin{align*} 
 &\Biggl[i\frac{\partial}{\partial t}- H_{\epsilon}(t), \Lambda(t)^2\Biggr]\Lambda(t)^{-2} = \Biggl[i\frac{\partial}{\partial t}- H_{\epsilon}(t), \Lambda(t)\Biggr]\Lambda(t)^{-1} \\
 & + \Lambda(t)\Biggl[i\frac{\partial}{\partial t}- H_{\epsilon}(t), \Lambda(t)\Biggr]\Lambda(t)^{-2} 
 =  \Qepsilon(t) + \Lambda(t)\Qepsilon(t)\Lambda(t)^{-1}.
 \end{align*}
 Hence, it follows from Lemma 4.1 and Proposition 3.3 that the assertion holds.  We consider the case $a = 3$.  From \eqref{4.6} we have 
 \begin{align*} 
 &\Biggl[i\frac{\partial}{\partial t}- H_{\epsilon}(t), \Lambda(t)^3\Biggr]\Lambda(t)^{-3} = \Biggl[i\frac{\partial}{\partial t}- H_{\epsilon}(t), \Lambda(t)\Biggr]\Lambda(t)^{-1} \\
 & + \Lambda(t)\Biggl\{\Biggl[i\frac{\partial}{\partial t}- H_{\epsilon}(t), \Lambda(t)^2\Biggr]\Lambda(t)^{-2}\Bigg\} \Lambda(t)^{-1} \\
 & = \Qepsilon(t) + \Lambda(t)Q_{2\epsilon}(t)\Lambda(t)^{-1}.
  \end{align*}
 Consequently, using the results for $a = 1, 2$, we see that the assertion holds.  In the same way we can prove the assertion for $a = 0,1,2,\dots$ by induction. 
 \par
 	Next we consider the case $a = -1,-2,\dots.$  From \eqref{4.6} we have 
 \begin{align*} 
0 =  &\Biggl[i\frac{\partial}{\partial t}- H_{\epsilon}(t), \Lambda(t)^{-1}\Lambda(t)\Biggr] = \Biggl[i\frac{\partial}{\partial t}- H_{\epsilon}(t), \Lambda(t)^{-1}\Biggr]\Lambda(t) \\
 & + \Lambda(t)^{-1}\Biggl[i\frac{\partial}{\partial t}- H_{\epsilon}(t), \Lambda(t)\Biggr],
 \end{align*}
 which shows 
 \begin{align*} 
 &\Biggl[i\frac{\partial}{\partial t}- H_{\epsilon}(t), \Lambda(t)^{-1}\Biggr]\Lambda(t) = - \Lambda(t)^{-1}\Biggl[i\frac{\partial}{\partial t}- H_{\epsilon}(t), \Lambda(t)\Biggr]\Lambda(t)^{-1}\Lambda(t) \\
 &  =  - \Lambda(t)^{-1}\Qepsilon(t)\Lambda(t).
 \end{align*}
 Hence the assertion for $a = -1$ holds.  We consider the case $a = -2$.  From \eqref{4.6} we have 
 \begin{align*} 
 &\Biggl[i\frac{\partial}{\partial t}- H_{\epsilon}(t), \Lambda(t)^{-2}\Biggr]\Lambda(t)^{2} = \Biggl[i\frac{\partial}{\partial t}- H_{\epsilon}(t), \Lambda(t)^{-1}\Biggr]\Lambda(t) \\
 & + \Lambda(t)^{-1}\Biggl\{\Biggl[i\frac{\partial}{\partial t}- H_{\epsilon}(t), \Lambda(t)^{-1}\Biggr]\Lambda(t)\Biggr\} \Lambda(t) \\
  & = Q_{-1\epsilon}(t) + \Lambda(t)^{-1}Q_{-1\epsilon}(t)\Lambda(t),
 \end{align*}
 which shows the assertion.  In the same way we can prove the assertion for $a = -1,-2,\dots$ by induction.  Therefore, our proof is complete.
 \end{proof}
 	We consider the equation
 \begin{equation} \label{4.7}
i\frac{\partial u}{\partial t}(t) = H_{\epsilon}(t)u(t) + f(t).
   \end{equation}
   \begin{pro}
   Let $u_0 \in B^a \ (\adots)$ and $f(t) \in \Eta$.  Then, under Assumption 2.1 there exist solutions $u_{\epsilon}(t) \in \mathcal{E}_t^1([0,T];B^a)\ (\zeo)$ with $u_{\epsilon}(0) = u_0$ to \eqref{4.7} satisfying
 \begin{equation} \label{4.8}
\sup_{\zeo}\Vert u_{\epsilon}(t)\Vert_a \leq C_a\left(\Vert u_{0}\Vert_a + \int_0^t \Vert f(\theta)\Vert_a d\theta\right).
   \end{equation}
   In particular, if $u_0 \in L^2$ and $f = 0$, we have $\Vert u_{\epsilon}(t)\Vert = \Vert u_{0}\Vert$.
   \end{pro}
   \begin{proof}
   Applying Theorem 2.5 in Chapter 2 of \cite{Kumano-go} to \eqref{4.1}, we see that each of $H_{\epsilon}(t)\ (\zeo)$ is written as the pseudo-differential operator with the symbol $p_{\epsilon}(t,x,\xi)$ satisfying
 \begin{equation*} 
\sup_{t,x,\xi}\left|p^{(\alpha)}_{\epsilon\ (\beta)} \txxi\right| \leq C_{\alpha\beta} < \infty
  \end{equation*}
  for all $\alpha$ and $\beta$, where $C_{\alpha\beta}$ may depend on $\zeo$.  Consequently, it follows from Lemma 2.5  of \cite{Ichinose 1995} with $s = a, a = 2(M+1)$ and $b = 2$ and \eqref{1.3} in the present paper that we have
  \begin{equation*}
  \sup_{0 \leq t \leq T}\Vert H_{\epsilon}(t)f\Vert_a \leq C_{a\epsilon}\Vert f\Vert_a
  \end{equation*}
  for $\adots$ with constants $C_{a\epsilon} \geq 0$ dependent on $\zeo$.  Hence, noting that
   the equation \eqref{4.7} is equivalent to 
 \begin{equation*} 
iu(t) = iu_0 + \int_0^t\bigl\{\Hepsilon(\theta)u(\theta) + f(\theta)\bigr\}d\theta,
  \end{equation*}
  we can find a solution $\uepsilon(t) \in \mathcal{E}_t^1([0,T];B^a)$ by the successive iteration for each $\zeo$.  From \eqref{4.5} and \eqref{4.7} we have
 \begin{equation} \label{4.9}
i\frac{\partial }{\partial t}\Lambda(t)^a\uepsilon(t) = H_{\epsilon}(t)\Lambda(t)^a\uepsilon(t) + Q_{a\epsilon}(t)\Lambda(t)^a\uepsilon(t) + \Lambda(t)^af(t).
   \end{equation}
   Applying the Calder\'on-Vaillancourt theorem (cf. p.224 in \cite{Kumano-go}),  from \eqref{3.25}, Propositions 3.3 and 3.5 we have
   $\Lambda(t)^a\uepsilon(t) \in \mathcal{E}_t^1([0,T];L^2)$ because of $\uepsilon(t) \in \mathcal{E}_t^1([0,T];B^a)$.  Noting that $H_{\epsilon}(t)$ defined by \eqref{4.1} is symmetric on $L^2$, from \eqref{4.9} we have
 \begin{align*} 
& \frac{d}{dt}\Vert \Lambda(t)^a\uepsilon(t)\Vert^2 = 2 \rittaire \left(\frac{\partial}{\partial t}\Lambda(t)^a\uepsilon(t),\Lambda(t)^a\uepsilon(t)\right) \\
& = -2\rittaire \bigl(iQ_{a\epsilon}(t)\Lambda(t)^a\uepsilon(t),\Lambda(t)^a\uepsilon(t)\bigr) - 2\rittaire \bigl(i\Lambda(t)^af(t),\Lambda(t)^a\uepsilon(t)\bigr).
  \end{align*}
Hence, using Proposition 4.2 and the Calder\'on-Vaillancourt theorem, we have
 \begin{align} \label{4.10}
 \frac{d}{dt}\Vert \Lambda(t)^a\uepsilon(t)\Vert^2 \leq 2C_a\left( \Vert \Lambda(t)^a\uepsilon(t)\Vert^2 + \Vert \Lambda(t)^af(t)\Vert \cdot\Vert \Lambda(t)^a\uepsilon(t)\Vert\right) 
  \end{align}
  with a constant $C_a$ independent of $\zeo$.   
  \par
  	For a moment take a constant $\eta > 0$ and set $v(t) := \bigl(\Vert\Lambda(t)^a\uepsilon(t)\Vert^2 + \eta\bigr)^{1_2}$, which is a positive and continuously differentiable  function with respect to $t$.  From \eqref{4.10} we have
 \begin{align*} 
 \frac{d}{dt}v(t)^2 \leq 2C_a\bigl( v(t)^2 + \Vert \Lambda(t)^af(t)\Vert v(t)\bigr) 
  \end{align*}
  and so $v'(t) \leq C_a\bigl( v(t) + \Vert \Lambda(t)^af(t)\Vert \bigr) $.  Hence we see
 \begin{align*} 
 v(t) \leq e^{C_at}v(0) + C_a\int_0^t e^{C_a(t-\theta)}\Vert\Lambda(\theta)^af(\theta)\Vert d\theta.
  \end{align*}
  Letting $\eta$ to $0$, we get
 \begin{align} \label{4.11}
 \Vert \Lambda(t)^a\uepsilon(t)\Vert \leq e^{C_at} \Vert \Lambda(0)^au_0\Vert + C_a\int_0^t e^{C_a(t-\theta)}\Vert\Lambda(\theta)^af(\theta)\Vert d\theta.  \end{align}
  Therefore, noting \eqref{3.25}, Propositions 3.3 and 3.5, we can prove \eqref{4.8} with another constants $C_a \geq 0$.
  \par
  	Finally, we consider the case where $a=0$ and $f=0$.  We have $d\Vert \uepsilon(t)\Vert^2/dt = 0$ instead of \eqref{4.10}, which shows $\Vert \uepsilon(t)\Vert = \Vert u_0\Vert$.
   \end{proof}
   The following has been proved in Lemma 3.1 of \cite{Ichinose 2012}.
   \begin{lem}
   Let $\adots$.  Then,
   the embedding map from $B^{a+1}$ into $B^{a}$ is compact.
   \end{lem}
	Now we will prove Theorem 2.1.  Our proof is similar to that of Theorem in \cite{Ichinose 1995}. 
	\par
	{\bf 1st step.}  Throughout 1st step we suppose $u_0 \in B^{a+1}$ and $f(t) \in \mathcal{E}_t^0([0,T]; \\
	B^{a+1})$.  Let $\uepsilon(t) \in \mathcal{E}_t^1([0,T];B^{a+1})\ (\zeo)$ be the solutions to \eqref{4.7} with $u(0) = u_0$ found in Proposition 4.3.  We see from \eqref{3.25}, Propositions 3.5 and 4.3 that the family $\bigl\{\uepsilon(t)\bigr\}_{\zeo}$
	is bounded in  $\mathcal{E}_t^0([0,T];B^{a+1})$ and 
equi-continuous in  $\mathcal{E}_t^0([0,T];B^{a})$ because 
 \begin{equation*} 
i\bigl\{\uepsilon(t)  - \uepsilon(t')\bigr\} = \int_{t'}^t\Hepsilon(\theta)\uepsilon(\theta)d\theta + \int_{t'}^tf(\theta)d\theta
  \end{equation*}
  and 
 \begin{equation*} 
\sup_{\zeo}\max_{0 \leq t \leq T}\Vert H_{\epsilon}(t)u_{\epsilon}(t)\Vert_a \leq C_a \sup_{\zeo}\max_{0 \leq t \leq T}\Vert u_{\epsilon}(t)\Vert_{a+1} \leq C'_a \Vert u_0\Vert _{a+1}.
  \end{equation*}
Consequently, it follows from Lemma 4.4 that we can apply the Ascoli-Arzel\`{a} theorem to $\left\{\uepsilon(t)\right\}_{\zeo}$ in $\Eta$.  Hence, there exist a sequence $\bigl\{\epsilon_j\bigr\}_{j=1}^{\infty}$ tending to zero and a function $u(t) \in \Eta$  such that 
 \begin{align} \label{4.12}
 \lim_{j \rightarrow \infty} u_{\epsilon_j}(t)  = u(t) \  \text{in}\ \Eta.
 \end{align}
  Then, since from \eqref{4.7} we have
 \begin{align*} 
& u_{\epsilon_j}(t) = u_0 -i \int_0^tH_{\epsilon_j}(\theta)u_{\epsilon_j}(\theta)d\theta - i \int_0^tf(\theta)d\theta \\
& = u_0 -i \int_0^tH_{\epsilon_j}(\theta)u(\theta)d\theta 
-i \int_0^tH_{\epsilon_j}(\theta)\bigl\{u_{\epsilon_j}(\theta)- u(\theta)\bigr\}d\theta - i \int_0^tf(\theta)d\theta,
  \end{align*}
  as in the proof of \eqref{3.14} in \cite{Ichinose 1995} we have
 \begin{equation*} 
u(t) = u_0 -i \int_0^tH(\theta)u(\theta)d\theta - i \int_0^tf(\theta)d\theta
  \end{equation*}
  in $\mathcal{E}_t^0([0,T];B^{a-1})$ by using Lemma 2.2 in \cite{Ichinose 1995}.  Therefore, we see that $u(t)$ belongs to $\Eta \cap \mathcal{E}_t^1([0,T];B^{a-1})$ and satisfies 
 \begin{equation} \label{4.13}
i\frac{\partial u}{\partial t}(t) = H(t)u(t) + f(t)
   \end{equation}
with $u(0) = u_0$.  From \eqref{4.8} and \eqref{4.12} we also have
 \begin{equation} \label{4.14}
\Vert u(t)\Vert_a \leq C_a\left(\Vert u_{0}\Vert_a + \int_0^t \Vert f(\theta)\Vert_a d\theta\right).
   \end{equation}
   \par
   {\bf 2nd step.}  In this step we will prove that a solution to \eqref{1.1} or \eqref{4.13} with a given initial data at $t=0$ is uniquely determined in $\Eta \cap \mathcal{E}_t^1([0,T];B^{a-1})$ for any $\adots$. 
   \par
   Let $u(t) \in \Eta \cap \mathcal{E}_t^1([0,T];B^{a-1})$ be a solution to \eqref{1.1}, i.e.
 \begin{equation*} 
i\frac{\partial u}{\partial t}(t) = H(t)u(t)
   \end{equation*}
   with $u(0) = 0$.  We may assume $a \leq 0$ because of $B^{a+1} \subset B^a$.  Let 
   $g(t) \in \mathcal{E}_t^0([0,T];B^{-a+2})$ be an arbitrary function and consider the equation
 \begin{equation*} 
i\frac{\partial v}{\partial t}(t) = H(t)v(t) + g(t)
   \end{equation*}
with $v(T) = 0$.  From the 1st step we can get a solution $v(t) \in \mathcal{E}_t^0([0,T];B^{-a+1}) \cap \mathcal{E}_t^1([0,T];B^{-a})$.  Then we have
\begin{align*}
0 &= \int_0^T\left(i\frac{\partial u}{\partial t}(t) - H(t)u(t), v(t)\right)dt \\
& =  \int_0^T\left(u(t), i\frac{\partial v}{\partial t}(t) - H(t)v(t)\right)dt = \int_0^T \bigl(u(t), g(t)\bigr)dt,
\end{align*}
which shows $u(t) = 0$.
\par
{\bf 3rd step.}  Let $u_0 \in B^a.$  We take $\bigl\{u_{0j}\bigr\}_{j=1}^{\infty}$ in $B^{a+1}$ such that $\lim_{j\to \infty}u_{0j} = u_0$ in $B^{a}$.  Let $u_j(t) \in \mathcal{E}_t^0([0,T];B^{a}) \cap \mathcal{E}_t^1([0,T];B^{a-1})$ be the solution to \eqref{1.1} with $u(0) = u_{0j}$, uniquely determined in the above 2 steps.  Since $u_j(t) - u_k(t) \in \mathcal{E}_t^0([0,T];B^{a}) \cap \mathcal{E}_t^1([0,T];B^{a-1})$ is the solution to \eqref{1.1} with $u(0) = u_{0j} - u_{0k} \in B^{a+1}$, from \eqref{4.14} we have
 \begin{equation} \label{4.15}
\Vert u_j(t) - u_k(t)\Vert_a \leq C_a\Vert u_{0j} - u_{0k}\Vert_a.
   \end{equation}
   Consequently, there exists a $u(t) \in \mathcal{E}_t^0([0,T];B^{a})$ such that $\lim_{j\to \infty}u_j(t) = u(t)$ in $\mathcal{E}_t^0([0,T];B^{a})$.  Since
 \begin{equation*} 
u_{j}(t) = u_{0j} -i \int_0^tH(\theta)u_{j}(\theta)d\theta,
  \end{equation*}
   $u(t)$ belongs to $\mathcal{E}_t^0([0,T];B^{a}) \cap \mathcal{E}_t^1([0,T];B^{a-1})$ and satisfies \eqref{1.1} with $u(0) = u_0$.  We can also prove \eqref{2.6} because $\Vert u_j(t)\Vert_a \leq C_a\Vert u_{0j}\Vert_a$ holds from \eqref{4.14}.
  \par
  	Finally, we will prove \eqref{2.7}.  Let $u_0 \in B^1$ and $\uepsilon(t)\ (\zeo)$ the solutions found in Proposition 4.3 to \eqref{4.7} with $u(0) = u_0$ and $f(t) = 0$.
	Then  as in the proof of  \eqref{4.14} we have $\Vert u(t)\Vert = \Vert u_0\Vert$ from $\Vert\uepsilon(t)\Vert = \Vert u_0\Vert$ and \eqref{4.12}.   Let $u_0 \in L^2$.  Take $\{u_{0j}\}_{j=1}^{\infty}$ in $B^1$ such that $\lim_{j \to \infty}u_{0j} = u_0$ in $L^2$ and let $u_j(t)$ be the solutions to \eqref{1.1} with $u(0) = u_{0j}$.  Then we have $\Vert u_j(t)\Vert = \Vert u_{0j}\Vert$.  Since we have proved $\lim_{j \to \infty}u_j(t) = u(t)$ in $\mathcal{E}_t^0([0,T];L^2)$ from \eqref{4.15}, we see \eqref{2.7}.
	Thus, our proof of Theorem 2.1 is complete. 
	\vspace{0.3cm}
\par
	Next, we will prove Theorem 2.2.  Our proof below is similar to that of Theorem 4.1 in \cite{Ichinose 2012}.
	\par
	Let $\utrho\ (\rho \in \mathcal{O})$ be the solutions to \eqref{1.1} with $u(0;\rho) = u_0 \in B^a \ (\adots)$.  Then, following the proof of Theorem 2.1, under the assumptions of Theorem 2.2 we have
 \begin{equation} \label{4.16}
\sup_{\rho \in \mathcal{O}}\Vert \utrho\Vert_a \leq C_a\Vert u_0\Vert_a\quad (0 \leq t \leq T).
   \end{equation}
   \par
   	We first assume $u_0 \in B^{a+1}$.  Then from \eqref{4.16} we have
 \begin{equation*} 
\sup_{\rho \in \mathcal{O}}\Vert \utrho\Vert_{a+1} \leq C_{a+1}\Vert u_0\Vert_{a+1}
   \end{equation*}
    and hence as in the 1st step of the proof of Theorem 2.1
 \begin{equation*} 
\Vert \utrho - u(t';\rho)\Vert_{a} \leq C'_{a}|t-t'|\Vert u_0\Vert_{a+1}
   \end{equation*}
   with a constant $C'_a$ independent of $\rho$.  Consequently, we see that the family $\big\{\utrho\bigr\}_{\rho \in \mathcal{O}}$ is bounded in $\mathcal{E}_t^0([0,T];B^{a+1})$ and equi-continuous in $\mathcal{E}_t^0([0,T];B^{a})$.  Let $\rho_j \to \rho$
   in $\mathcal{O}$ as $j \to \infty$.  Noting Lemma 4.4, we can apply the Ascoli-Arzel\`{a} theorem to $\big\{u(t;\rho_j)\bigr\}_{j=1}^{\infty}$ in $\mathcal{E}_t^0([0,T];B^{a})$.  Then, there exist a subsequence $\big\{j_k\bigr\}_{k=1}^{\infty}$ and a function $v(t) \in \Eta$ such that $\lim_{k\to \infty}u(t;\rho_{j_k}) = v(t) $ in $\Eta$.  As in the proof of \eqref{4.13} we see that $v(t)$ belongs to $\mathcal{E}_t^1([0,T];B^{a-1})$ and satisfies \eqref{1.1} with $u(0) = u_0$.  The uniqueness of solutions to \eqref{1.1} gives $v(t) = \utrho$, which shows $\lim_{k\to \infty}u(t;\rho_{j_k}) = \utrho$.  Using the uniqueness of solutions to \eqref{1.1} again,  we have
 \begin{equation*} 
\lim_{j \to \infty}u(t;\rho_{j}) = \utrho\ \text{in}\ \Eta.
   \end{equation*}
   Therefore, we see that the mapping $: \mathcal{O} \ni \rho \to \utrho \in \Eta$ is continuous.
   \par
   Now let  $u_0 \in B^a$ and $\utrho\ (\rho \in \mathcal{O})$ the solutions to \eqref{1.1} with $u(0) = u_0$.   We take $\bigl\{u_{0k}\bigr\}_{k=1}^{\infty}$ in $B^{a+1}$ such that $\lim_{k\to \infty}u_{0k} = u_0$ in $B^{a}$ and let $u_k(t;\rho) \in \mathcal{E}_t^0([0,T];B^{a}) \cap \mathcal{E}_t^1([0,T];B^{a-1})$ be the solutions to \eqref{1.1} with $u(0) = u_{0k}$. Then, from \eqref{4.16} we have
 \begin{equation*} 
\sup_{\rho}\max_t \Vert u_k(t;\rho) - u(t;\rho)\Vert_a \leq C_a\Vert u_{0k} - u_{0}\Vert_a,
   \end{equation*}
which shows that $\utrho$ is continuous in $\Eta$ with respect to $\rho \in \mathcal{O}$, because so is $ u_k(t;\rho)$.  Thus, our proof of Theorem 2.2 is complete. \vspace{0.3cm} 
\par
	In the end of this section we will prove Theorem 2.3.  Our proof below is similar to that of Theorem 2.3 in \cite{Ichinose 2012}.
	\par
	Let $u_0 \in B^a\ (\adots)$ and $f(t) \in \Eta$.  Then, we see that under Assumption 2.1 there exists the unique solution $u(t) \in \Eta \cap \mathcal{E}_t^1([0,T];B^{a-1})$ to \eqref{4.13} with $u(0) = u_0$, which satisfies 
 \begin{equation} \label{4.17}
\Vert u(t)\Vert_a \leq C_a\left(\Vert u_{0}\Vert_a + \int_0^t \Vert f(\theta)\Vert_a d\theta\right).
   \end{equation}
   Its proof can be completed by using \eqref{4.14} as in the 3rd step of the proof of Theorem 2.1.
	\par
	Let $u_0 \in B^{a+1}$ and $\utrho \in \mathcal{E}_t^0([0,T];B^{a+1}) \cap \mathcal{E}_t^1([0,T];B^{a})\ (\rho \in \mathcal{O})$ the solutions to \eqref{1.1} with $u(0) = u_0$.  Let $\rho \in \mathcal{O}$ be fixed and $\tau \not= 0$ a small constant such that $\rho + \tau \in \mathcal{O}.$  We set
 \begin{equation} \label{4.18}
w_{\tau} (t;\rho) := \frac{u(t;\rho+\tau) - u(t;\rho)}{\tau},
 \end{equation}
   which belongs to $\Eta \cap \mathcal{E}_t^1([0,T];B^{a-1})$.
   Then, we have $w_{\tau}(0;\rho) = 0$ and from \eqref{1.1}
\begin{equation} \label{4.19}
 i\frac{\partial}{\partial t}w_{\tau} (t;\rho) = H(t;\rho)w_{\tau} (t;\rho) + \int_0^1 \frac{\partial H}{\partial \rho}(t;\rho+\theta\tau)d\theta\, u(t;\rho+\tau).
 \end{equation}
 Hence from \eqref{4.16} and \eqref{4.17} we get
\begin{align*}
& \Vert w_{\tau} (t;\rho)\Vert_a \leq C_a\int_0^t\int_0^1 \left\Vert \frac{\partial H}{\partial\rho}(t';\rho+\theta\tau)u(t';\rho+\tau)\right\Vert_ad\theta dt' \\
& \leq C'_a\int_0^t \left\Vert u(t';\rho+\tau)\right\Vert_{a+1} dt' \leq C''_a \left\Vert u_0\right\Vert_{a+1}.
\end{align*}
Consequently, 
 \begin{equation} \label{4.20}
\sup_{\tau}\Vert w_{\tau} (t;\rho)\Vert_a \leq C_a \left\Vert u_0\right\Vert_{a+1}
\end{equation}
   with another constant $C_a$.
   \par
     We first assume $u_0 \in B^{a+2}$.  From  \eqref{4.20} we have
 \begin{equation*} 
\sup_{\tau}\Vert w_{\tau} (t;\rho)\Vert_{a+1} \leq C_{a+1} \left\Vert u_0\right\Vert_{a+2}.
\end{equation*}
  Thereby from \eqref{4.16} and \eqref{4.19} we have  
   \begin{equation*} 
\sup_{\tau}\Vert w_{\tau} (t;\rho) - w_{\tau} (t';\rho)\Vert_{a} \leq C'_a |t - t'|\left \Vert u_0\right\Vert_{a+2}
\end{equation*}
as in the 1st step of the proof of Theorem 2.1. Hence, we can apply the Ascoli-Arzel\`{a} theorem to $\bigl\{w_{\tau}(t;\rho)\bigr\}_{\tau}$ in $\Eta$.  In addition, we can use  the uniqueness of solutions to \eqref{2.13} or \eqref{4.13}.  Then, using Theorem 2.2, 
as in the 3rd step of the proof of Theorem 2.1 we can prove  from \eqref{4.19} that there exists a function $\wtrho \in \Eta \cap \mathcal{E}_t^1([0,T];B^{a-1})$ satisfying \eqref{2.13} with $w(0) = 0$ and 
 \begin{equation} \label{4.21}
\lim_{\tau\to 0}w_{\tau} (t;\rho) = \wtrho \ \text{in}\ \Eta.
\end{equation}
\par
	Now let  $u_0 \in B^{a+1}$.  Let $\utrho$ be the solution to \eqref{1.1} with $u(0) = u_0$ and  $w_{\tau}(t;\rho)$ be defined by \eqref{4.18}.  We take $\bigl\{u_{0k}\bigr\}_{k=1}^{\infty}$ in $B^{a+2}$ such that $\lim_{k\to \infty}u_{0k} = u_0$ in $B^{a+1}$.  Let $u_k(t;\rho) \in \mathcal{E}_t^0([0,T];B^{a+1}) \cap \mathcal{E}_t^1([0,T];B^{a})$ be the solution to \eqref{1.1} with $u(0) = u_{0k}$.  We define $w_{k\tau}$ and $w_{k}$ by \eqref{4.18} and \eqref{4.21} for $u = u_k$, respectively.  From \eqref{4.19} we have
	\begin{align*}
 & i\frac{\partial}{\partial t}\bigl\{w_{k\tau} (t;\rho) - w_{\tau} (t;\rho)\bigr\}= 
 H(t;\rho)\bigl\{w_{k\tau} (t;\rho) - w_{\tau} (t;\rho)\bigr\} \\
 &\quad  + \int_0^1 \frac{\partial H}{\partial \rho}(t;\rho+\theta\tau)d\theta\, \bigl\{u_k(t;\rho+\tau) - u (t;\rho+\tau)\bigr\}
 \end{align*}
 and $w_{k\tau} - w_{\tau} \in \Eta \cap \mathcal{E}_t^1([0,T];B^{a-1}) $.  Hence, using \eqref{4.17}, from  \eqref{4.16}  we have
 \begin{equation} \label{4.22}
\sup_{\tau}\Vert w_{k\tau}(t;\rho) - w_{\tau}(t;\rho)\Vert_a \leq C_a\Vert u_{0k} - u_{0}\Vert_{a+1}.
   \end{equation}
As noted in the early part of this proof of Theorem 2.3, there exists
 the solution $\wtrho \in \mathcal{E}_t^0([0,T];B^{a}) \cap \mathcal{E}_t^1([0,T];B^{a-1})$  to \eqref{2.13} with $w(0) = 0$ because of $\partial_{\rho}H(t;\rho)u(t;\rho) \in \Eta$.    Then, as in the proof of \eqref{4.22} from \eqref{2.13} we  have
 \begin{equation} \label{4.23}
 \Vert w_{k}(t;\rho) - w(t;\rho)\Vert_a \leq C_a\Vert u_{0k} - u_{0}\Vert_{a+1}
   \end{equation}
with constants $C_a \geq 0$ independent of $\rho \in \mathcal{O}$
because both of $w_k$ and $w$ are the solutions to \eqref{2.13}.
Consequently, we have
\begin{align*}
& \Vert w_{\tau} (t;\rho)- w(t;\rho)\Vert_a \leq \Vert w_{\tau} - w_{k\tau}\Vert_a + \Vert w_{k\tau}- w_k\Vert_a + \Vert w_{k} - w\Vert_a \\
& \leq 2C_a \Vert u_{0k} - u_{0}\Vert_{a+1} + \Vert w_{k\tau}- w_k\Vert_a.
\end{align*}
Therefore, we see from \eqref{4.21} that for any $\epsilon > 0$ we get  $\overline{\lim}_{\tau \to 0}\max_{\,t}\Vert w_{\tau}- w\Vert_a < \epsilon,$ which shows
 \begin{equation} \label{4.24}
\lim_{\tau \to 0}\max_{0 \leq t \leq T}\Vert w_{\tau}(t;\rho) - w(t;\rho)\Vert_a = 0.
   \end{equation}
We also have \eqref{2.12} from \eqref{4.20} and \eqref{4.24}.
\par
	In the end of this proof we will prove that $\wtrho = \partial_{\rho}\utrho$ for $u_0 \in B^{a+1}$ is continuous in $\mathcal{E}_t^0([0,T];B^{a})$ with respect to $\rho \in \mathcal{O}$.  
	We first assume $u_0 \in B^{a+2}$.  Then we have \eqref{4.16} where $a$ is replaced by $a+2$.  Since $\wtrho$ is the solution to \eqref{2.13} with $w(0) = 0$, we see from \eqref{4.17} as in the proof of \eqref{4.21} that the family $\bigl\{\wtrho\bigr\}_{\rho \in \mathcal{O}}$ is bounded in $\mathcal{E}_t^0([0,T];B^{a+1})$ and equi-continuous in $\Eta$. 
Hence, noting that $\utrho$ is continuous in $\mathcal{E}_t^0([0,T];B^{a+2})$  with respect to $\rho$, we see that so is $\wtrho$ in $\Eta$ as in the proof of Theorem 2.2.  Now let $u_0 \in B^{a+1}$.  We take $\bigl\{u_{0k}\bigr\}_{k=1}^{\infty}$ in $B^{a+2}$ such that $\lim_{k\to \infty}u_{0k} = u_0$ in $B^{a+1}$ and write  as $w_k(t,\rho)$ the solutions to \eqref{2.13} with $u(t;\rho) = u_k(t;\rho)$ and $w(0) = 0$.  Then we have \eqref{4.23}, which shows that $\wtrho$ is continuous with respect to $\rho \in \mathcal{O}$ in $\Eta$.  Therefore, our proof of Theorem 2.3 is complete.
%
   \section{Proofs of Theorems 2.4 - 2.6}
	In this section we will study the 4-particle systems \eqref{1.4}.   Let $(x,\xi) \in \bR^{2d}$ and write 
\begin{equation} \label{5.1}
 h_k(t,x,\xi) :=  \frac{1}{2m_k}|\xi - A^{(k)}(t,x)|^2 + V_k(t,x) \ (k = 1,2,3,4)
 \end{equation}
  and 
\begin{equation} \label{5.2}
 l_k(x,\xi) :=  \frac{1}{2m_k}|\xi |^2 + <x>^2 \ (k = 3,4).
 \end{equation}
  We set 
\begin{equation} \label{5.3}
 \tilde{h}(t,z,\zeta) :=  \sum_{k=1}^2 h_k(t,x^{(k)},\xi^{(k)}) + W_{12}(t,x^{(1)} -x^{(2)}) + \sum_{k=3}^4l_k(x^{(k)},\xi^{(k)})
 \end{equation}
 and write 
\begin{equation} \label{5.4}
 \widetilde{H}(t) :=  \widetilde{H}\left(t,\frac{Z+Z'}{2},D_z\right),
 \end{equation}
 where $z = (x^{(1)},x^{(2)},x^{(3)},x^{(4)})$ and $\zeta = (\xi^{(1)},\xi^{(2)},\xi^{(3)},\xi^{(4)})$ in $\bR^{4d}$.   We  also set 
\begin{equation} \label{5.5}
\tilde{h}_s(t,z,\zeta) :=   \tilde{h}(t,z,\zeta) + i\sum_{k=1}^2\frac{1}{2m_k}\nabla\cdot A^{(k)}(t,x^{(k)})
 \end{equation}
 and 
\begin{equation*} 
 p_{\mu}(t,z,\zeta) :=  \frac{1}{\mu + \tilde{h}_s(t,z,\zeta)}
 \end{equation*}
 for large $\mu$ as in \eqref{3.2} and \eqref{3.3}, respectively.
 \begin{lem}
 Assume \eqref{1.2}, \eqref{2.1} and \eqref{2.3} - \eqref{2.4} for $(V_k,A^{(k)})\ (k = 1,2)$ with $M = M_k$ and \eqref{2.14} - \eqref{2.15} for $W_{12}$.  Then, there exist a constant $\mu^* \geq 0$ and functions $r_{\mu}(t,z,\zeta)\ (\mu \geq \mu^*)$ such that
 \begin{equation} \label{5.5-2}
	\mu^* + \rittaire	\, \tilde{h}_s(t,z,\zeta) \geq C_0^*(<\zeta>^2 + \Phi(z)^2),
	\end{equation}
	\begin{equation} \label{5.6}
	\bigl[\mu + \widetilde{H}(t)\bigr] P_{\mu}\tzdz = I + R_{\mu}\tzdz,
	\end{equation}
	\begin{equation} \label{5.7}
	\left|r^{(\alpha)}_{\mu\ (\beta)}(t,z,\zeta)\right| \leq  C_{\alpha\beta}\, \mu^{-1/2}
	\end{equation}
	in $[0,T]\times \bR^{8d}$ for all $\alpha, \beta$ and $\mu \geq \mu^*$ with constants $C_0^* > 0$ and $C_{\alpha\beta}$ independent of $\mu$, where
	\begin{equation} \label{5.8}
	\Phi(z) =  \sum_{k=1}^2<x^{(k)}>^{M_k+1} + \sum_{k=3}^4<x^{(k)}>.
	\end{equation}
 \end{lem}
 \begin{proof}
 As in the proof of  \eqref{3.6} we see
 \begin{equation*} 
  \rittaire \tilde{h}_s(t,z,\zeta) = \tilde{h}(t,z,\zeta) \geq  C_0\bigl(<\zeta>^2  + \Phi(z)^2\bigr) - |W_{12}(t,x^{(1)} -x^{(2)})| - C_1
   \end{equation*}
   with constants $C_0 > 0$ and $C_1 \geq 0$.  Hence, uisng the assumption \eqref{2.14}, we can determine constants $\mu^* \geq 0$ and $C^*_0 > 0$ satisfying \eqref{5.5-2}.  Then, using \eqref{5.5-2},
    as in the proof of \eqref{3.7} for $\mu \geq \mu^*$ we have
   \begin{align} \label{5.10}
  & r_{\mu}(t,z,\zeta)   =  \sum_{|\alpha| = 1} \int_0^1 d\theta\  \text{Os}-\iint e^{-iy\cdot\eta}
 <y>^{-2l_0}  <D_{\eta}>^{2l_0}<\eta>^{-2l_1} <D_y>^{2l_1} \notag\\
 & \hspace{2cm} \cdot \tilde{h}_s^{(\alpha)}(t,z,\zeta+\theta\eta)  p_{\mu (\alpha)}(t,z+y,\xi)dy\hdbar\eta
 \end{align}
  for large integers $l_0$ and $l_1$.  In addition, as in the proofs of \eqref{3.8} - \eqref{3.9} we can show 
 \begin{equation} \label{5.11}
 |\tilde{h}^{(\alpha)}_{s\,(\beta)}(t,z+y,\zeta)| \leq  C_{\alpha\beta} \left(<\zeta>^2 + \Phi(z+y)^2\right) 
 \end{equation}
 for all $\alpha$ and $|\beta|\geq 1$, and 
 \begin{equation} \label{5.12}
 |\tilde{h}^{(\alpha)}_{s\, (\beta)}(t,z,\zeta + \theta\eta)| \leq  C_{\alpha\beta} \bigl(<\zeta> + \Phi(z)\bigr) <\eta>
 \end{equation}
 for $|\alpha|\geq 1$ and all $\beta$.  Therefore, we can complete the proof of Lemma 5.1 from \eqref{5.10} - \eqref{5.12} as in the proof of Lemma 3.2.
 \end{proof}
 	We can easily see from \eqref{5.11} and \eqref{5.12} as in the proof of \eqref{3.25} that under the assumptions of Lemma 5.1 we have
 \begin{equation} \label{5.12-2}
 |\tilde{h}^{(\alpha)}_{s\, (\beta)}(t,z,\zeta)| \leq  C_{\alpha\beta}\bigl(<\zeta>^2 + \Phi(z)^2\bigr) 
 \end{equation}
 for all $\alpha$ and  $\beta$.
 \begin{pro}
 Under the assumptions of Lemma 5.1 there exist a constant $\mu \geq \mu^*$ and a function $w(t,z,\zeta)$ satisfying 
 \begin{equation} \label{5.13}
  |w^{(\alpha)}_{\ \  (\beta)}(t,z,\zeta)| \leq C_{\alpha\beta}\left(<\zeta>^2 + \Phi(z)^2\right)^{-1}
  \end{equation}
  for all $\alpha, \beta$ and 
 \begin{equation} \label{5.14}
 W\tzdz = \bigl(\mu+ \widetilde{H}(t)\bigr)^{-1}.
   \end{equation}
 \end{pro}
 \begin{proof}
If $\mu \geq \mu^*$, from \eqref{5.5-2} and  \eqref{5.11} - \eqref{5.12} we see
 \begin{equation*} 
 |p^{(\alpha)}_{\mu\, (\beta)}(t,z,\zeta)| \leq  C_{\alpha\beta}\bigl(<\zeta>^2 + \Phi(z)^2\bigr)^{-1}
 \end{equation*}
 for all $\alpha$ and  $\beta$ as in the proof of Proposition 3.3. Hence, using Lemma 5.1, we can prove Proposition 5.2 as in the proof of Proposition 3.3.
 \end{proof}
	We take a $\mu$ in Proposition 5.2 and fix it hereafter.  We set
 \begin{equation} \label{5.15}
 \lambda(t,z,\zeta)  := \mu + \tilde{h}_s(t,z,\zeta)
   \end{equation}
   as in \eqref{3.13}.  Then, from \eqref{5.1} - \eqref{5.5} we have 
 \begin{align} \label{5.16}
 \Lambda(t) & = \Lambda\tzdz = \mu + \widetilde{H}_s\tzdz =  \mu + \widetilde{H}(t) \notag\\
 & = \mu + H_1(t) + H_2(t) + W_{12}(t) + L_3(t) + L_4(t),
  \end{align}
   where $H_k(t)$ are the operators defined by \eqref{1.4} and $L_k(t)$ the pseudo-differential operators with symbols $l_k(x^{(k)},\xi^{(k)})$ defined by \eqref{5.2}.  Using the real-valued function $\tilde{h}(t,z,\zeta)$ defined by \eqref{5.3}, we determine
 \begin{equation} \label{5.17}
 \chi_{\epsilon}(t,z,\zeta) = \chi\bigl(\epsilon(\mu + \tilde{h}(t,z,\zeta)\bigr)
   \end{equation}
   with constants $\zeo$ and $\chi \in \mathcal{S}(\bR)$ such that $\chi(0) = 1$ as in \eqref{3.15}.
   \par
   	Lemmas 5.3 and 5.4 below are crucial in this section.
	\begin{lem}
	Under the assumptions of Lemma 5.1 there exist functions  $k_{\epsilon}(t,z,\zeta)\ (\zeo)$ in $[0,T] \times \bR^{8d}$ satisfying
 \begin{equation} \label{5.18}
 \sup_{0 < \epsilon \leq 1}\sup_{t,z,\zeta} |k^{(\alpha)}_{\epsilon\  (\beta)}(t,z,\zeta)| \leq C_{\alpha\beta} < \infty
  \end{equation}
  for all $\alpha, \beta$ and 
 \begin{equation} \label{5.19}
K_{\epsilon}\tzdz = \Bigl[X_{\epsilon}\tzdz, \Lambda\tzdz\Bigr].
   \end{equation}
	\end{lem}
	\begin{proof}
	As in the proof of \eqref{3.18} we see
   \begin{align} \label{5.20}
  & k_{\epsilon}(t,z,\zeta) = \sum_{|\alpha| = 1} \Bigl\{\chi_{\epsilon}^{(\alpha)}(t,z,\zeta)\lambda_{(\alpha)}(t,z,\zeta) 
  - \lambda^{(\alpha)}(t,z,\zeta)\chi_{\epsilon(\alpha)}(t,z,\zeta) \Bigr\}\notag\\
 &\quad   + 2 \sum_{|\gamma| = 2}\frac{1}{\gamma\, !}\int_0^1 (1 - \theta)d\theta\  \text{Os}-\iint e^{-iy\cdot\eta}
 \Bigl\{\chi_{\epsilon}^{(\gamma)}(t,z,\zeta+\theta\eta)\lambda_{(\gamma)}(t,z+y,\zeta)  \notag\\
 & \qquad \quad 
  - \lambda^{(\gamma)}(t,z,\zeta+\theta\eta)\chi_{\epsilon(\gamma)}(t,z+y,\xi) \Bigr\}dy\hdbar\eta 
  \equiv I_{1\epsilon} + I_{2\epsilon}.
 \end{align}
From \eqref{5.5}, \eqref{5.15} and \eqref{5.17} we can write
   \begin{align} \label{5.21}
  & I_{1\epsilon}(t,z,\zeta) = \epsilon\chi\,'(\epsilon(\mu +\tilde{h}))\sum_{|\alpha| = 1}\Bigl\{\tilde{h}^{(\alpha)}\tilde{h}_{s(\alpha)}
  - \tilde{h}^{(\alpha)}_{s}\tilde{h}_{(\alpha)} \Bigr\}\notag\\
 & = i \epsilon\chi\,'(\epsilon(\mu +\tilde{h}))\sum_{|\alpha| = 1}
 \tilde{h}^{(\alpha)}_s(t,z,\zeta)\sum_{k=1}^2 \frac{1}{2m_k}(-i\partial_z)^{\alpha}\,\nabla\cdot A^{(k)}\bigl(t,x^{(k)},\xi^{(k)}\bigr).
 \end{align}
   From \eqref{5.5-2} we have
 \begin{equation} \label{5.21-2}
 \bigl (\mu + \tilde{h}(t,z,\zeta)\bigr)^{-1} \leq  C_0\bigl(<\zeta>^2  + \Phi(z)^2\bigr)^{-1}
   \end{equation}
 because of $\tilde{h} = \rittaire \tilde{h}_s$.  Hence, from \eqref{2.4}, \eqref{5.12} and \eqref{5.21} - \eqref{5.21-2} we can prove $\sup_{\epsilon}\sup_{t,z,\zeta} 
   |I_{1\epsilon}| < \infty$ as in the proof of \eqref{3.20}.  In the same way we can prove
 \begin{equation} \label{5.22}
 \sup_{0 < \epsilon \leq 1}\sup_{t,z,\zeta} |I^{(\alpha)}_{1\epsilon\  (\beta)}(t,z,\zeta)| \leq C_{\alpha\beta} < \infty
  \end{equation}
for all $\alpha$ and $\beta$.  \par
   Let $|\gamma| = 2$.  Then, from \eqref{5.5} and \eqref{5.11} - \eqref{5.12} we have the similar inequalities
 \begin{equation*} 
\sup_{\zeo}  |\chi^{(\alpha+\gamma)}_{\epsilon\  (\beta)}(t,z,\zeta)| \leq C_{\alpha\beta}\left(<\zeta>^2 + \Phi(z)^2\right)^{-1}
  \end{equation*}
  and 
 \begin{equation*} 
\sup_{\zeo}  |\chi^{(\alpha)}_{\epsilon\ (\beta+\gamma)}(t,z,\zeta)| \leq C_{\alpha\beta} < \infty
  \end{equation*}
  to \eqref{3.21} and \eqref{3.22} for all $\alpha$ and $\beta$, respectively.  Consequently, noting that $\lambda^{(\gamma)}(t,z,\zeta) = \tilde{h}_s^{(\gamma)}(t,z,\zeta)$ are constants, from \eqref{5.20} we can prove
 \begin{equation*} 
 \sup_{0 < \epsilon \leq 1}\sup_{t,z,\zeta} |I^{(\alpha)}_{2\epsilon\  (\beta)}(t,z,\zeta)| \leq C_{\alpha\beta} < \infty
  \end{equation*}
for all $\alpha$ and $\beta$ as in the proof of Lemma 3.4, which completes the proof together with \eqref{5.20} and \eqref{5.22}.
	\end{proof}
Let $H(t)$ be the operator defined by \eqref{1.4}.
\begin{lem}
Besides the assumptions of Lemma 5.1 we suppose that each $(V_k,A^{(k)})\ (k = 3,4)$ satisfies \eqref{1.6} - \eqref{1.7} and each 
$W_{ij}(t,x)\ (1 \leq i <j \leq 4)$ except $W_{12}$ satisfies \eqref{1.6}.  Then,  there exists a function $\tilde{q}(t,z,\zeta)$  satisfying
 \begin{equation} \label{5.23}
 \sup_{t,z,\zeta} |\tilde{q}^{\,(\alpha)}_{\ \,(\beta)}(t,z,\zeta)| \leq C_{\alpha\beta} < \infty
  \end{equation}
  for all $\alpha, \beta$ and 
 \begin{equation} \label{5.24}
\widetilde{Q}\tzdz = \Bigl[\Lambda(t), H(t)\Bigr]\Lambda(t)^{-1}.
   \end{equation}
\end{lem}
\begin{proof}
We write $H(t)$   as
 \begin{equation} \label{5.25}
 H(t) = \sum_{k=1}^4H_k(t) + W_{12}(t) + \sum \,'\,W_{ij}(t).
   \end{equation}
  Then from \eqref{5.16} we see
 \begin{align} \label{5.26}
&[H(t), \Lambda(t)] = [(H_1 + H_2 + W_{12}) + H_3 + H_4 + \sum\, '\,W_{ij}, \notag\\
& (H_1 + H_2 + W_{12}) + L_3 + L_4] = [H_3,L_3] + [H_4,L_4] + \notag\\
& \left[\sum\, '\,W_{ij}, H_1 + H_2 + L_3 + L_4\right].
   \end{align}
  Lemma 3.1 in \cite{Ichinose 1995} has showed that each of $[H_k,L_k] \Lambda^{-1}\ (k = 3,4) $ is written as the pseudo-differential operator with the symbol satisfying \eqref{5.23}.
  \par
  We can easily see that $m_1\bigl[W_{13}(t),H_1(t)\bigr]$ is written as the pseudo-differential operator with the symbol
 \begin{align} \label{5.27}
& \tilde{q}_1(t,z,\zeta)  = i\frac{\partial W_{13}}{\partial x} (t,x^{(1)} - x^{(3)})\cdot \xi^{(1)} + \frac{1}{2}\Delta_xW_{13} (t,x^{(1)} - x^{(3)}) \notag \\
& - i A^{(1)}(t,x^{(1)})\cdot \frac{\partial W_{13}}{\partial x} (t,x^{(1)} - x^{(3)}).
  \end{align}
  Hence, from the assumptions we have
 \begin{align*} 
& |\tilde{q}_1(t,z,\zeta)| \leq C_1\bigl(<\xi^{(1)}>^2 + <x^{(1)} - x^{(3)}>^2 + <x^{(1)}>^{M_1+1}<x^{(1)} - x^{(3)}>\bigr)
\notag \\
& \leq C_2\bigl(<\xi^{(1)}>^2 + <x^{(1)}>^{M_1+2} + <x^{(1)}>^{2(M_1+1)}+ < x^{(3)}>^2\bigr) \\
& \leq C_3\bigl(<\xi^{(1)}>^2 + <x^{(1)}>^{2(M_1+1)}+ < x^{(3)}>^2\bigr).
  \end{align*}
  In the same way we have
 \begin{equation} \label{5.28}
 |\tilde{q}^{\,(\alpha)}_{1 \,(\beta)}(t,z,\zeta)| \leq C_{\alpha\beta} \bigl(<\zeta>^2 + \Phi(z)^2 \bigr)
  \end{equation}
  for all $\alpha$ and $ \beta$.  Consequently, by Proposition 5.2 we see that $\bigl[W_{13}(t),H_1(t)\bigr]\Lambda(t)^{-1}$ is written as the pseudo-differential operator with the symbol satisfying \eqref{5.23}.  In the same way  we can complete the proof of Proposition 5.4.
\end{proof}
	Using the function $\chi_{\epsilon}(t,z,\zeta)$ defined by \eqref{5.17},  we define 
 \begin{equation} \label{5.29}
H_{\epsilon}(t) := X_{\epsilon}\tzdz^{\dagger}H(t)X_{\epsilon}\tzdz
   \end{equation}
as in \eqref{4.1}.
\begin{lem}
Under Assumption 2.2 
 there exist functions $\qepsilon(t,z,\zeta)\ (\zeo)$ satisfying \eqref{5.18} and 
 \begin{align} \label{5.30}
Q_{\epsilon}\tzdz = &\Bigl[\Lambda\tzdz, H_{\epsilon}(t)\Bigr]\Lambda\tzdz^{-1} \notag \\
&  + i\frac{\partial\Lambda}{\partial t}\tzdz\Lambda\tzdz^{-1}.
   \end{align}
\end{lem}
\begin{proof}
From \eqref{5.29} and $\Lambda(t)^{\dagger} = \Lambda(t)$ we have
 \begin{align*}
 & \Bigl[\Lambda(t), H_{\epsilon}(t)\Bigr] = -\Bigl[\Lambda(t), X_{\epsilon}(t)\Bigr]^{\dagger}H(t)X_{\epsilon}(t) \\
 &\quad  + X_{\epsilon}(t)^{\dagger}\Bigl[\Lambda(t), H(t)\Bigr]X_{\epsilon}(t) + 
 X_{\epsilon}(t)^{\dagger}H(t)\Bigl[\Lambda(t), X_{\epsilon}(t)\Bigr]
 \end{align*}
 as in the proof of Lemma 4.1.  Let's apply Proposition 5.2 and Lemmas 5.3 - 5.4 to the above and apply Proposition 5.2 to $(i\partial \Lambda(t)/\partial t)\Lambda(t)^{-1}$.  Then, we can prove Lemma 5.5 as in the proof of Lemma 4.1.
\end{proof}
	Using Lemma 5.5, we can prove the following as in the proof of Proposition 4.2.
	\begin{pro}
Under Assumption 2.2 there exist functions $q_{a\epsilon}(t,z,\zeta)\ (\adots, \zeo)$ satisfying \eqref{5.18} and 
 \begin{equation} \label{5.31}
Q_{a\epsilon}\tzdz = \Biggl[i\frac{\partial}{\partial t}- H_{\epsilon}(t), \Lambda(t)^a\Biggr]\Lambda(t)^{-a}.
   \end{equation}
	\end{pro}
		Let $B'^a(\bR^{4d})\ (\adots)$ be the weighted Sobolev spaces introduced in \S 2.  Then we see as in Lemma 4.4 that the embedding map from $B'^{a+1}$ into $B'^{a}$ is compact.  We also get the similar result to Proposition 3.5 from Proposition 5.2.  Therefore, using Proposition 5.6, we can prove Theorems 2.4 - 2.6 as in the proofs of Theorems 2.1 - 2.3, respectively.	
 %

\end{document}